\documentclass[11pt]{amsart}
\usepackage{amsmath}
\usepackage{upgreek}
\usepackage{xcolor}
\usepackage[all]{xy}
\usepackage{graphicx}
\usepackage[cp1251]{inputenc}
\usepackage{bm}
\usepackage{enumerate}
\theoremstyle{plain}
\newtheorem{theorem}{Theorem}[section]
\newtheorem{lemma}[theorem]{Lemma}
\newtheorem{prop}[theorem]{Proposition}

\newcommand{\td}{\text{d}}
\theoremstyle{definition}
\newtheorem{remark}[theorem]{Remark}

\numberwithin{equation}{section}

\gdef\x{\xi}

\newcommand{\R}{\mathbb R}
\newcommand{\Z}{\mathbb Z}

\newcommand{\C}{\mathbb C}

\newcommand{\im}{\mathrm{im}}
\begin{document}




\title[Abelian instantons over the Chen-Teo AF geometry]{Abelian instantons over the Chen-Teo AF geometry}

\author{Thomas John Baird}
\author{Hari K. Kunduri}


\thanks{T.J. Baird and H.K  Kunduri acknowledge the support of NSERC Grants  RGPIN-2016-05382 and RGPIN-2018-04887 respectively.}

\begin{abstract}
We classify finite energy harmonic 2-forms on the asymptotically flat gravitational instanton constructed by Chen and Teo. We prove that every $U(1)$-bundle admits a unique anti-self-dual Yang-Mills instanton (up to gauge equivalence) which we describe explicitly in coordinates.  As an application, we compute the classical partition function for Maxwell theory with theta term.
\end{abstract}
\maketitle
\section{Introduction}\label{intro}

The classic black hole uniqueness theorem in general relativity states that the set of asymptotically flat (or AF) stationary and axisymmetric black hole solutions of the vacuum Einstein equations  is exhausted by the two parameter family of Kerr solutions (for precise hypotheses, see \cite[Thm 3.2]{Chrusciel:2012jk}).  

In 1980, Lapedes \cite{L} conjectured that an analogue of this Lorentzian result holds in Riemannian geometry.  Consider the homogeneous space $M_\flat := \R^4/ \psi$, the quotient of flat $\R^4 = \R^3 \times \R$ under an isometry $\psi$ which acts by a rotation of $\R^3$ and a non-trivial translation of $\R$. An \textbf{AF gravitational instanton} is a complete, Ricci-flat Riemannian manifold $(M,g)$ that approaches $M_\flat$ asymptotically at infinity. This terminology is due to an analogy in quantum gravity with Yang-Mills instantons in quantum gauge theory \cite{Gibbons:1976ue}. 

Notice that $M_\flat$ admits a $T^2 = U(1) \times U(1)$ isometry generated by coaxial rotations $\frac{\partial}{\partial \phi}$ of $\R^3$ and translations $\frac{\partial}{\partial \tau}$ of $\R$. Lapedes conjectured that the only non-trivial  AF gravitational instantons with a 2-torus isometry are the `Euclidean Kerr' manifolds (see \eqref{EKerr}).

A counterexample was constructed thirty one years later by Chen and Teo \cite{CT1,CT2} (it necessarily does not admit a Lorentzian section). The Chen-Teo AF instanton $(M,g)$ has underlying manifold diffeomorphic to the complex projective plane with circle removed $M \cong \C P^2 \setminus S^1$. Under the $T^2$ action $(M,g)$ has three nuts, two finite bolts, and two infinite bolts within the classification framework of \cite{Gibbons:1979xm}.

The aim of the present work is to classify Abelian Yang-Mills instantons on $(M,g)$. Since $M$ is simply connected, such solutions are in one-to-one correspondence with (anti-)self-dual $L^2$ harmonic 2-forms that satisfy the `Dirac quantization condition', i.e., its periods over all 2-cycles lie in $2\pi \Z$. We first classify finite energy harmonic 2-forms on $(M,g)$, which we show are spanned by a self-dual form and two anti-self-dual forms which we construct explicitly. Essential to our construction is exploiting the $T^2$ isometry. In particular, we derive and solve a pair of scalar PDEs whose solutions generate self-dual and anti-self-dual forms respectively. 

By integrating these forms over 2-cycles, we identify those harmonic forms satisfying Dirac quantization, and prove that every $U(1)$-bundle $P$ over $M$ admits a unique Abelian instanton up to gauge equivalence, which we identify explicitly in coordinates. As an application, we calculate the semi-classical partition function for Maxwell Theory with theta term \cite{T, W}.

Harmonic 2-forms on gravitational instantons were intensively studied in the late 90s  in relation to Sen's conjecture \cite{Sen}. Gibbons \cite{Gibbons} constructed a class of such harmonic 2-forms on the Taub-NUT geometry, by observing that Killing vector fields naturally give rise to harmonic forms and metrics which admit such Killing fields with bounded norm give rise to finite-energy harmonic forms.  Hitchin \cite{H} subsequently showed this spanned the entire space of $L^2$ harmonic forms as part of his proof of Sen's S-duality conjecture in that setting. This approach was also used by Etesi and Hausel to construct Abelian Yang-Mills instantons for the Euclidean Schwarzschild metric \cite{EH}.  We also refer the reader to related work of Franchetti on  $L^2$ harmonic forms over asymptotically locally flat gravitational instantons of type $A_{K-1}$ and $D_K$ \cite{F} and on the Taub-bolt geometry \cite{F2}. 

Our paper in outline:

\begin{itemize}
\item[\S \ref{Lchf}] We prove there are exactly three independent $L^2$ harmonic 2-forms on $(M,g)$. 

\item[\S \ref{CHF}] We reduce the problem to a scalar equation in three dimensions and use the Gibbons construction to generate two solutions, $\omega_\pm$, on $(M,g)$.

\item[\S \ref{CTmet}] We present the Chen-Teo metric and identify the third solution $\omega_2$. This was obtained indirectly from a solution to a related scalar PDE found by Bossard, Katmadas, and Turton \cite{GKT}. 

\item[\S \ref{Period integrals}] We use fixed point localization  to compute the periods of these solutions over 2-cycles and identify the instantons.  

\item[\S \ref{INMPF}] We use Poincar\'e duality to compute intersection pairings and evaluate the classical partition function for Maxwell theory with theta term.

\end{itemize}

We include as well five appendices. 

\begin{itemize}

\item[\S \ref{ALFintro}] An introduction to AF and ALF metrics.

\item[\S \ref{3Dr}] Further details on reducing to 3 dimensions supplementing \S \ref{CHF}.

\item[\S \ref{Smoothness of the functions (x,y)}] Verification that our solutions are globally smooth.

\item[\S \ref{A homeomorphism from M}] A diffeomorphism between $\C P^2 \setminus S^1$ and the Chen-Teo manifold. 

\item[\S \ref{LThf}] An alternative calculation of the periods and intersection numbers using a reduction to the two-dimensional orbit space $M \setminus T^2$.

\end{itemize}

\section{$L^2$ cohomology and harmonic forms}\label{Lchf}

Given a complete Riemannian manifold $(M,g)$, the degree $k$ \emph{harmonic cohomology}, denoted   $\mathcal{H}^k(M,g)$, is the vector space of $L^2$-integrable harmonic $k$-forms. That is 
$$  \mathcal{H}^k(M,g) := \{ \omega \in \Omega^k(M;\R) |  \td\omega = \td \star \omega =0,  ~\int_M \omega \wedge \star \omega < \infty\}.   $$  

\begin{lemma}\label{2only}
If $(M,g)$ is a complete Riemannian 4-manifold of infinite volume with non-negative Ricci curvature, then $\mathcal{H}^k(M,g) =0$ unless $k=2$.
\end{lemma}
\begin{proof}
Harmonic $0$-forms must be constant, so since $(M,g)$ has infinite volume $ \mathcal{H}^0(M,g) =0$. Because $\mathrm{Ric}(g) \geq 0$ and $g$ is complete, $L^2$ harmonic forms must be parallel (Yau \cite{Y} Thm 6), hence zero since $(M,g)$ has infinite volume, so $\mathcal{H}^1(M,g) =0$. Hodge duality implies that $\mathcal{H}^4(M,g) = \mathcal{H}^0(M,g)=0$ and $\mathcal{H}^3(M,g) = \mathcal{H}^1(M,g)=0$.
\end{proof}

\begin{prop}
For the Chen-Teo instanton $(M,g)$, we have $ \mathcal{H}^2(M,g) \cong  \R^3$.
\end{prop}

\begin{proof}
From Lemma \ref{2only} it follows that the dimension of $\mathcal{H}^2(M,g)$ is equal to the $L^2$-Euler characteristic
$$ \chi_{(2)}(M,g) := \sum_{i=0}^4 (-1)^i\dim \mathcal{H}^i(M,g) . $$

This can be calculated using a result of Carron \cite{C}.  Since $(M,g)$ is asymptotically flat, it is quasi-isomorphic to a metric $g'$ which is flat outside of a compact set. Since Hodge cohomology on complete manifolds is invariant under quasi isomorphism (see for example \cite{HHM} \S 2.1), it follows that 
$$ \chi_{(2)}(M,g) = \chi_{(2)}(M,g'). $$
Finally, since the end of $(M,g')$, has torsion free fundamental group $ \pi_1(\R \times S^2 \times S^1) \cong \mathbb{Z}$ it follows from Carron (\cite{C} Thm. 1.7) that $\chi_{(2)}(M,g')$ agrees with the topological Euler characteristic of $M \cong \C P^2 \setminus S^1$, which is easily shown to equal 3. Alternatively, it equals the integral of the Euler form, which was shown to equal 3 by Chen and Teo \cite{CT1}.
\end{proof}

The Hodge star operator $\star$ determines an order two automorphism of $\mathcal{H}^2(M,g)$, so we have a decomposition into self-dual and anti-self-dual forms
$$ \mathcal{H}^2(M,g)  =  \mathcal{H}^2_+(M,g) \oplus \mathcal{H}^2_-(M,g)$$
where $\mathcal{H}^2_{\pm}(M,g) := \{ \omega \in \mathcal{H}^2(M,g)|  \star \omega = \pm \omega\}$.

Because the harmonic cohomology $\mathcal{H}^2(M,g)$ is naturally isomorphic to the $L^2$-cohomology, there are natural linear maps
\begin{equation}\label{commdiag}
  H^2_c(M) \stackrel{i}{\rightarrow} \mathcal{H}^2(M,g) \stackrel{j}{\rightarrow} H^2(M).
\end{equation}
where $ H^2_c(M)$ and  $H^2(M)$ are compactly supported and ordinary de Rham cohomology respectively and the composition $j \circ i$ is the natural map (see \cite{SS} Lemma 1.3). Since $M \cong \C P^2 \setminus S^1$  we have 
\begin{equation}
H^2(M) \cong H^2_c(M) \cong \R^2 
\end{equation} 
and by Poincar\'e duality a non-degenerate pairing \begin{eqnarray}\label{pairing333}
H^2(M) \otimes H^2_c(M) \rightarrow \R  &&( \alpha, \beta ):=  \int \alpha \wedge \beta.
\end{eqnarray}

The compactification $\overline{M} \cong \C P^2$ satisfies 
\begin{equation}
H^2(\overline{M}) = H^2_c(\overline{M}) \cong \R,
\end{equation} 
and we have a commuting diagram
\begin{equation}\label{commdiag2}
\xymatrix{   &  H^*_c(\overline{M})  = H^*(\overline{M}) \ar[dr] & \\   H^*_c(M) \ar[ur] \ar[rr]^{j\circ i} & & H^*(M) }
\end{equation}
where the diagonal arrows are the natural maps.  From (\ref{commdiag2}) it can be verified that $\im(j \circ i)$ coincides with the image of $H^*(\overline{M}) \rightarrow H^*(M)$ so in particular, $j \circ i$ has rank one. In \S \ref{Period integrals} we describe $i$ and $j$ explicitly showing they both have rank two.

The symmetric pairing  
\begin{eqnarray}\label{symmpair}
H^2_c(M) \otimes H^2_c(M) \rightarrow \R 
\end{eqnarray} 
restricting (\ref{pairing333}) is semi-definite of rank one. We will see in Proposition \ref{finen} that under our orientation convention, 
\begin{eqnarray}
 \mathcal{H}^2_+(M,g) \cong \R & & \mathcal{H}^2_-(M,g) \cong \R^2 
\end{eqnarray} 
which implies (\ref{symmpair}) is negative definite. This implies for instance that bolts have non-positive self intersection number.

\section{Constructing harmonic forms on a manifold with a local isometry}\label{harmonic}\label{CHF}
\subsection{Harmonic forms associated to isometries} 
On a complete, Ricci-flat Riemannian manifold $(M,g)$ admitting an isometry generated by a Killing vector field, it is simple to obtain a harmonic two-form using the following elementary construction (see, e.g.~\cite{Gibbons}.)  Using the identity 
\begin{equation}
\td \star \td K = 2 \text{Ric}(K, \sim) \;, 
\end{equation} where we denote by $K$ both the vector field $K$ and its metric dual $g(K, ~)$. We see immediately for a Ricci flat space, $\td \star \td K = 0$, which implies that $\td K$ is co-closed. Since it is trivially closed and $(M,g)$ is complete,  $\td K$ is a harmonic 2-form.  Consequently  
\begin{equation}\label{GibbonsForms}
\omega^\pm   \equiv \td K  \pm \star \td K
\end{equation} are respectively self-dual and anti-self-dual harmonic 2-forms. 

The $L^2$ energy of these forms are given by
\begin{equation}
\frac{1}{2}|| \td K||^2_{L^2} = \frac{1}{2} || \star \td K||^2_{L^2} \equiv \frac{1}{8\pi^2}\int_M \td K\wedge \star \td K.
\end{equation}
In the Chen-Teo metric~\eqref{g}, choose $K = \frac{\partial}{\partial \tau}$. It can be checked explicitly (see Proposition \ref{finen} ) that as $r \to \infty$, 
\begin{equation}
\omega^\pm \wedge \star \, \omega^\pm = 2 \td K \wedge \star\, \td K \pm 2 \td K \wedge \td K =  O(r^{-2}) \td r \wedge \td \theta \wedge  \td \psi \wedge  \td \phi
\end{equation}  
$\omega_+$ and $\omega_-$ are respectively self-dual and anti-self-dual $L^2$ harmonic forms and therefore provide two of three basis vectors for $\mathcal{H}^2(M,g)$.
A similar construction involving the Killing vector $ \partial_\phi$  of \eqref{g} gives rise to harmonic forms whose $L^2$ energy density grows like $O(r^2)\td r \wedge \td \theta \wedge  \td \psi \wedge  \td \phi$,  which clearly diverges.

\subsection{The reduction to three dimensions} \label{Reduction}
The vacuum Einstein equation, $\text{Ric}(g) =0$, for a Riemannian 4-manifold $(M,g)$ admitting a Killing vector field $K$ can locally be reduced to a system of equations in 3 dimensions as follows. In a local neighbourhood where $K \neq 0$, it is possible to introduce local coordinates $\tau, x^i$, $i=1,2,3$ so that $K = \partial_\tau$ and $L_Kx^i  =0$ and the metric $g$ can then be written in the form:
\begin{equation}
g = \lambda(\td\tau + \Omega)^2 + \frac{h}{\lambda}
\end{equation} where $\Omega$ is 1-form and $h$ is a $K$-invariant Riemannian metric in the transverse coordinates $x_1,x_2,x_3$, and $\lambda :=g(K,K) =  |K|_g^2$. We will define an orientation on $(M,g)$ by choosing $\td \text{Vol}(g) = \lambda^{-1} \td \tau \wedge \td \text{Vol} (h)$ to have positive orientation. Introducing the globally defined \emph{twist one-form} 
\begin{equation}
 T := \star_g (K \wedge \td K)
 \end{equation}  which is also a form on the orbit space (indeed note that $i_K T =0$) , the condition $\text{Ric}(g) =0$ is equivalent to the system
\begin{equation} \label{Ricciflat}
\begin{aligned}
R_{ab}(h) &= \frac{1}{2\lambda^2} \partial_a \lambda \partial_b \lambda - \frac{1}{2\lambda^2} T_a T_b \\
\Delta_h \lambda &= \frac{|\td \lambda|^2}{\lambda} + \frac{|T|^2}{\lambda} \Leftrightarrow \td \star_h \td \lambda = \frac{\td \lambda \wedge \star_h \td \lambda}{\lambda} + \frac{T \wedge \star_h T}{\lambda} \\
\td T & = 0 \;, \qquad \td \star_h (\lambda^{-2} T) =0 \Leftrightarrow \nabla^a \left(\frac{T_a}{\lambda^2}\right) =0.
\end{aligned}
\end{equation}  Here $\star_h$, $\Delta_h$, and $\nabla$  refer to the Hodge dual, Laplacian, and Riemannian connection associated to $h$.  We also use $|\alpha|^2 = h^{ij} \alpha_i \alpha_j/p!$ to refer to the inner product on a $p-$form $\alpha$ induced by $h$. 

Since $\td T =0$, we can write $T = \td\zeta$ for some locally defined potential $\zeta = \zeta(x^i)$. In the Chen-Teo case, $M$ is simply connected so $\zeta$ is globally defined. Finally, $T$ is related to $\Omega$ locally by
\begin{equation}
T = -\lambda^2 \star_h \td\Omega , 
\end{equation} which  is consistent with the $\td\star_h (\lambda^{-2} T) =0$ equation.   Note that $\zeta$ must satisfy  
\begin{equation}\label{Laptau}
\Delta_h \zeta = \frac{2}{\lambda} \td\zeta \cdot \td\lambda. 
\end{equation}  where $\cdot$ denotes the inner product on forms induced by $h$. 

It is well known that the system $(h, \lambda, \zeta)$  is equivalent to a 3d theory of gravity coupled to a `non-linear sigma model'. The latter in mathematician's language is a harmonic map with a target space that is a homogeneous space. This means that $(\lambda, \zeta)$ parameterize the 2d manifold $SL(2, R)/ SO(1,1)$ with a Lorentzian signature target space equipped with the Anti-de Sitter metric, which has constant negative curvature.  In the better known setting in which $g$ has Lorentzian signature that is relevant to axisymmetric vacuum solutions of general relativity, the target is $SL(2,R)/ SO(2) = \mathbb{H}^2$ equipped with its canonical constant negative curvature metric. 

We turn now to the (anti-)self-duality equations. By a result of Hitchin \cite{H},  if $|K|$ grows at most linearly in $|x|$ at infinity, then $\omega \in \mathcal{H}^2(M,g)$ must be invariant under $K$. We have
\begin{equation}\label{Liediv0}
L_K \omega = L_K \star \omega=0 
\end{equation} which implies, $i_K \omega = i_K \star \omega = E$ for some closed form $E$. In particular, if $M$ is simply connected, $E  = \td \alpha$ for a globally defined function $\alpha$.  Restrict to an open set on which $|K|^2 >0$ (in the Chen-Teo metric this could correspond to the Killing field $\partial_\tau$ in the standard chart). Then (\ref{Liediv0}) can be inverted to yield
\begin{equation}\label{omegaKilling}
\omega = \frac{1}{\lambda} (K \wedge \td\alpha + \star (K \wedge \td\alpha) )  = (\td\tau +\Omega) \wedge \td\alpha + \star\left[(\td\tau + \Omega)\wedge \td\alpha \right]
\end{equation} which is manifestly self-dual. We have
\begin{eqnarray*}
\td\omega & =& \td \Omega \wedge \td\alpha + \td \star \left[(\td\tau + \Omega) \wedge \td\alpha\right]\\
&=& -\frac{1}{\lambda^2}  \td\alpha \wedge \star_h T + \td\left[ \frac{1}{\lambda} \star_h \td\alpha \right ] \\
&=&  -\frac{1}{\lambda^2}  \td\alpha \wedge \star_h T - \frac{1}{\lambda^2} \td\lambda \wedge \star_h \td\alpha + \frac{1}{\lambda} \td \star_h \td\alpha
\end{eqnarray*}
so the condition $\td\omega =0$ is equivalent to
\begin{equation}
\td \star_h \td \alpha - \frac{1}{\lambda} \td \alpha \wedge \star_h (T + \td \lambda) =0 .
\end{equation} This is equivalent to the following elliptic second order PDE for $\alpha$: 
\begin{equation}\label{scalarSD}
\Delta_h \alpha - \frac{1}{\lambda} \td(\lambda+\zeta) \cdot_h \td\alpha  =0.
\end{equation}
Similarly the anti-self-dual form $(\td\tau +\Omega) \wedge \td\alpha - \star(\td\tau + \Omega)\wedge \td\alpha $ is closed if and only if
\begin{equation}\label{scalarASD}
 \Delta_h \alpha  -\frac{1}{\lambda} \td (\lambda -\zeta) \cdot_h \td\alpha  =0.
 \end{equation}
 It is natural to associate an  $L^2$ energy density $E[\alpha]$ to (anti)-self-dual two forms:
\begin{eqnarray}\label{L2cond}
\omega \wedge \star \omega =  \frac{2}{\lambda} |\td\alpha|_g^2 \; \td \text{Vol}(g) = 2 |\td\alpha|^2 \; \td \text{Vol}(g) := E[\alpha]\, .
\end{eqnarray} 
\begin{prop} \label{Gibbons} Consider a four-dimensional Ricci flat metric admitting a local one-parameter family of isometries generated by the Killing vector field $K$. Then $\alpha = \alpha_+ := \lambda + \zeta$ satisfies  \eqref{scalarSD}  and $\alpha = \alpha_- := \lambda - \zeta$ satisfies \eqref{scalarASD} where $\lambda = g(K,K)$ and $T = \star_g (K \wedge \td K)$.
\end{prop}
\begin{proof} We have from \eqref{Ricciflat} and \eqref{Laptau}
\begin{equation}
\Delta_h (\lambda \pm \zeta) = \frac{|\td \lambda|^2}{\lambda} + \frac{|T|^2}{\lambda} \pm \frac{2}{\lambda} \td \zeta \cdot \td \lambda = \frac{1}{\lambda}| \td \lambda \pm \td\zeta|^2.
\end{equation}
\end{proof} 
\begin{remark} The harmonic forms generated by the smooth functions $\alpha_{\pm} =  \lambda \pm \zeta$ are precisely the harmonic forms $\omega_{\pm}$  (\ref{GibbonsForms}). To see this, note the identity, which holds for the one-form $K$ associated to a Killing vector field:
\begin{equation}
\td K = -\frac{1}{\lambda} \left[ \star (K \wedge T) + K \wedge \td \lambda\right]
\end{equation} This is derived by a simple computation using the definition of the twist form $T$. We then find that 
\begin{equation}
\begin{aligned}
\td K \pm \star \td K &= - \frac{1}{\lambda} \left[ \star (K \wedge T) + K \wedge \td \lambda \pm \star (K \wedge \td \lambda) \pm K \wedge T \right] \\
& = -\frac{1}{\lambda} \left[ K \wedge \td (\lambda \pm \zeta) \pm \star (K \wedge \td (\lambda \pm  \zeta))\right].
\end{aligned} 
\end{equation} where we used $T = \td \zeta$.  As observed above, if $\text{Ric}(g) =0$, the above (anti-)self-dual form is automatically harmonic. We easily read off this corresponds to the choice $\alpha_\pm = \lambda \pm \zeta$ . 
\end{remark}
\begin{remark} The (anti-)self-dual forms generated by solutions of \eqref{scalarSD} and \eqref{scalarASD} respectively can be expressed as the curvatures of locally defined gauge potentials as follows.   We may rewrite the PDEs \eqref{scalarSD}, \eqref{scalarASD} in the form of a conserved current
\begin{equation}\label{conserved}
\td \star_h \left[ \frac{\td \alpha}{\lambda} \mp \frac{\alpha \td \zeta}{\lambda^2} \right] =0
\end{equation} where we used the fact $\td \star_h (\lambda^{-2} \td \zeta)=0$ which follows from the reduced system \eqref{Ricciflat}.  This implies the existence of locally defined one-forms $P_\pm$ satisfying
\begin{equation} \label{potential}
\td P_\pm =\star_h \left[ \frac{\td \alpha}{\lambda} \mp \frac{\alpha \td \zeta}{\lambda^2} \right] . 
\end{equation} This allows for the construction of local connection one-forms 
\begin{equation}
\td A^\pm = \frac{1}{\lambda} \left( K \wedge \td \alpha \pm \star (K \wedge \td \alpha)\right)
\end{equation} where 
\begin{equation}\label{gaugefield}
A^\pm := -\alpha (\td \tau + \Omega) \pm P_\pm .
\end{equation}
\end{remark}

Equation (\ref{scalarASD}) was considered by \cite{GKT} (up to a Moebius transformation) in the context of constructing six-dimensional supergravity solutions fibering over the Chen-Teo metric. They found a second local solution  (equation (2.38)  \cite{GKT}) which corresponds to solution (\ref{alpha2}). We will show this local solution extends to a global, finite energy solution and therefore determines the last remaining basis vector for $\mathcal{H}^2(M,g)$.

\section{The Chen-Teo metric on $\mathbb{C}P^2 \setminus S^1$}\label{CTmet}

Recently Chen-Teo, using the integrability properties of the vacuum Einstein equations restricted to $T^2$-invariant solutions,  constructed a new two-parameter family of  AF gravitational instantons on $\mathbb{C}P^2 \setminus S^1$ \cite{CT1}. The metric is cohomogeneity-two and admits a torus action as isometries, thus producing an explicit counterexample to Lapades' conjecture \cite{L}.  We review properties of the Chen-Teo instanton in the following section. 

\subsection{The local metrics} Consider the following family of Ricci flat metrics $g$  given in local coordinates $(\tau, x, y, \phi)$ as follows \cite{CT2}:
The local form of the metric is:
\begin{equation}\label{g}
g = \frac{F(x,y)}{(x-y)H(x,y)} \left( \td\tau + \frac{G(x,y)}{F(x,y)} \td\phi\right)^2 + \frac{\kappa H(x,y)}{(x-y)^3} \left( \frac{\td x^2}{X(x)} - \frac{\td y^2}{Y(y)} - \frac{X(x)Y(y)}{\kappa F(x,y)} \td\phi^2\right)
\end{equation} where $X = P(x), Y = P(y)$ are quartic polynomials 
\begin{equation}
P(u) = a_0 + a_1 u + a_2 u^2 + a_3 u^3 + a_4 u^4
\end{equation} and $F(x,y), H(x,y), G(x,y)$ are polynomials in $x,y$ given by
\begin{equation}
\begin{aligned}
F(x,y) & = y^2 X - x^2 Y\\
H(x,y) & =(\nu x + y)\left[(\nu x - y)(a_1 - a_3 x y) - 2(1 - \nu)(a_0 -a_4 x^2 y^2)\right] \\
G(x,y) & = (\nu^2 a_0 + 2 \nu a_3 y^3 + 2 \nu a_4 y^4 - a_4 y^4)X + (a_0 - 2\nu a_0 - 2\nu a_1 x - \nu^2 a_4 x^4) Y
\end{aligned}
\end{equation} In what follows, we will omit the explicit dependence on the coordinates $x,y$ and simply write $F, H, G$, etc. The whole solutions has 7 parameters: the 5 $a_i$ and two other parameters $\nu$ and $\kappa$.  The local metric \eqref{g} is a solution to the positive-signature vacuum Einstein equations, i.e. 
\begin{equation}
\text{Ric}(g) =0, 
\end{equation} for any choice of the $(a_i, \nu, \kappa)$ . Two of the $a_i$ can be fixed by using scaling  symmetries of the metric acting on the $a_i$ and $(x,y)$, leaving a five-parameter family of local metrics. The remaining constants $a_i$ are not totally arbitrary, since we will impose that the quartic $P(u)$ to have four real roots, corresponding to fixed point sets of the $T^2$ action generated by the commuting Killing vector fields $\frac{\partial}{\partial \tau}, \frac{\partial}{\partial \phi}$. Note that $\kappa$ is just an overall scaling (analogous to the parameter $m$ of Schwarzschild) so we should be able to fix this without loss of generality.   Assuming the parameters and coordinates are chosen such that $H >  0$ and $x>y$) in this chart, we have
\begin{equation}
\sqrt{\det g} = \frac{ \kappa H}{(x-y)^5}. 
\end{equation}  As explained in Section \ref{Reduction} we can perform a reduction to three dimensions by  choosing 
\begin{equation}
K = \frac{\partial}{\partial \tau}
\end{equation} to read off the metric $h$ on the three-dimensional space of orbits: 
\begin{equation}
h = \frac{\kappa F}{(x-y)^4} \left( \frac{\td x^2}{X} - \frac{\td y^2}{Y} - \frac{XY}{\kappa F} \td\phi^2\right)\; ,  \qquad \sqrt{\det h} = \frac{\kappa F}{(x-y)^6}. 
\end{equation} We read off
\begin{equation}
|K|^2 = \lambda = \frac{F}{(x-y) H}, \qquad \Omega = \frac{G}{F} d\phi
\end{equation}  and it can be checked explicitly that the twist form $T$ is indeed closed:
\begin{equation}
T = -\lambda^2 \star_h \td\Omega  =- \frac{F^2 \partial_x \Omega}{H^2 Y} \td y  - \frac{F^2 \partial_y \Omega}{H^2 X} \td x
\end{equation} and as discussed above, we may write $T = \td\zeta$ for some locally defined function $\zeta$ . 
A computation reveals that
\begin{align}
\lambda &=  \frac{(x-y) (\nu x + y ) (a_1 - a_3 x y)}{2(\nu -1) H} - \frac{x+y}{2(\nu -1)(\nu x + y)} \\
\zeta & = - \frac{(x-y) (\nu x + y ) (a_1 - a_3 x y)}{2(\nu -1) H}  -\frac{x+y}{2(\nu -1)(\nu x + y)} 
\end{align}
where we have fixed an integration constant in writing down $\zeta$.  Note that the Laplacian of a scalar function $f$ with respect to $h$ is given by
\begin{equation}
\Delta_h f = \frac{(x-y)^6}{F} \left[ \partial_x \left( \frac{F h^{xx}}{(x-y)^6} \partial_x f \right) + \partial_y \left( \frac{F h^{yy}}{(x-y)^6}  \partial_y f \right) \right]
\end{equation}  As explained in \S \ref{Reduction} we immediately find that the functions
\begin{equation}\label{alphapm}
\alpha_+ = -\frac{x+y}{(\nu -1)(\nu x + y)}, \qquad \alpha_- = \frac{(x-y) (\nu x + y ) (a_1 - a_3 x y)}{(\nu -1) H}
\end{equation} generate a self-dual and anti self-dual harmonic forms $\omega^\pm$  according to Proposition \ref{Gibbons}. As mentioned above, these harmonic forms arise automatically as a consequence of $\text{Ric}(g) =0$ and a local isometry.   Remarkably, we have explicitly checked that
\begin{equation}\label{alpha2}
\alpha_2 = \frac{\alpha_-}{(a_1 - a_3 x y )} = \frac{(x-y) (\nu  x + y)}{(\nu -1) H}
\end{equation} produces another solution to \eqref{scalarASD} and hence generates  a second anti self-dual harmonic 2-form.  As explained at the end of \S \ref{Reduction}, we have obtained this solution rather indirectly from the construction of certain solutions of six-dimensional supergravity admitting Killing spinors \cite{GKT}.  

The locally defined gauge fields \eqref{gaugefield} associated to the (anti-)self-dual harmonic forms $\omega_{\pm}, \omega_2$ can be obtained by explicitly integrating \eqref{potential} in the above coordinate chart.  If we define the smooth function
\begin{equation}
Q:= \frac{ x(y + x \nu)(a_1 - a_3 x y )Y}{y (1-\nu)(x - y)(a_0 (x+y) - x y (x y (a_3 + a_4(x+y)) - a_1))} ,
\end{equation} then the gauge fields associated to $\omega_\pm$ can be expressed as
\begin{equation}\label{Apm}
A^\pm = -\alpha_\pm (\td \tau + \Omega) \pm \left(Q + \frac{a_3 y^2 - a_1 \nu}{y(1 - \nu)}\right)  \td \phi
\end{equation} whereas the gauge field associated to $\omega_2$ is given by
\begin{equation}\label{A2}
A^-_2 = -\alpha_2 (\td \tau + \Omega) - \left( \frac{Q}{a_1 - a_3 x y}  - \frac{\nu}{y(1 - \nu)}\right) \td \phi. 
\end{equation}  The connection 1-forms \eqref{Apm}, \eqref{A2} have been expressed in a gauge in which $\td \star A =0$. 
\subsection{Global Analysis} The seven-parameter family of local metrics \eqref{g} is sufficiently general that it can be extended to global metrics on various manifolds that are asymptotically flat, asymptotically locally flat, or asymptotically locally Euclidean~\cite{CT2}.  We are interested in the subfamily that extends to a globally smooth, asymptotically flat metric on $\mathbb{C}P^{2} \setminus S^1$~\cite{CT1, CT2}.  The proof of this  was carried out in the original parameterization \cite{CT1},  although in the second parametrization \cite{CT2} of the solution used here,  not all restrictions on the parameters are given and so we review them below.   To ensure positive-definite signature of the metric, we will restrict $x>y$ and impose the conditions $F, H > 0$ and $\kappa >0$, as well as $ Y< 0 < X$.   The metric $g$ degenerates at the zeroes of $X, Y$ and we will see below these correspond to fixed point sets of the torus action.  An explicit calculation of the Riemann tensor of $g$ reveals that $\text{Riem}(g) \to 0$ as $ x - y \to 0^+$.  We identify this region as an asymptotically flat end of $(M,g)$.   In this end, 
\begin{equation}\label{K2inf}
|K|^2 \to \frac{1}{1 -\nu^2} \qquad \text{as } x-y \to 0^+. 
\end{equation} Hence provided $|\nu| < 1$, the norm of the distinguished Killing vector field $K$ approaches a finite constant value as required (the cases $\nu = \pm 1$ in fact correspond to asymptotically locally Euclidean metrics).   The local metric admits a scaling symmetry so that $\nu \in [-1,1]$ can always be arranged \cite{CT2}. We will in addition restrict to $|\nu| < 1$ in what follows.  Note that the $(x,y)$ chart degenerates in this limit and spherical coordinates must be introduced in this end.  We will show  below that $g$ indeed approaches the model geometry $g_0$ \eqref{model}.

The two-dimensional subspaces of the tangent space orthogonal to the commuting Killing vector fields  are integrable as a consequence of $\text{Ric}(g) =0$ and Frobenius' theorem. The space of these orbits are parameterized by the coordinates $x,y$.  Indeed the metric $g$ \eqref{g} is of the general Weyl-Papapetrou form \cite{Hollands:2007aj}
\begin{equation}
g = \hat{g}_{ij} \td x^i \td x ^j + G_{AB} \td \xi^A \td \xi^B
\end{equation} where $x^i = (x,y)$ and $\xi^i = (\tau, \phi)$ and the functions $\hat{g}_{ij}, G_{AB}$ are independent of $\xi^i$.  In particular, $G_{AB}$ is simply the restriction of $g$ to the generators of the torus action:
\begin{equation}
\tilde{g} = \frac{F}{(x-y) H} \left(\td \tau + \frac{G}{F} \td \phi \right)^2 + \frac{ H X (-Y)}{(x-y)^3 F} \td \phi^2. 
\end{equation} Setting $\rho = \sqrt{\det \tilde g}$ to be the  `area density' of the orbits of the isometry group (note that $\det \tilde g \geq 0$) one finds explicitly that 
\begin{equation}\label{rho}
\rho^2 = -\frac{X Y}{(x-y)^4}. 
\end{equation} It is a straightforward exercise to show, using $\text{Ric}(g) =0$, that $\rho$ is harmonic on the orbit space $(B, \hat{g})$  where $B:= M \setminus T^2$, i.e. $\Delta_{2} \rho =0$ where $\Delta_{2}$ is the Laplacian associated to $g_2$.  Its harmonic conjugate, defined by $\td z = \star_2 \td \rho$ is given by (after fixing an integration constant)
\begin{equation}\label{z}
z = \frac{2(a_0 + a_2 x y + a_4 x^2 y^2) + (x+y)(a_1 + a_3 x y)}{2(x-y)^2}. 
\end{equation}  Under the assumption that there are no $p \in M$ with a discrete isotropy subgroup, $B$ can to shown  to be \cite[Prop. 1]{Hollands:2007aj} a two-dimensional simply connected manifold with boundary and corners, i.e. a manifold locally modelled over $\R \times \R$ (interior points) , $\R_+ \times \R$ (one-dimensional boundary segments) and $\R_+ \times \R_+$ (corners).   In the particular case of $M = \mathbb{C}P \setminus S^1$ the requirement of no discrete isotropy subgroups in the interior of $B$ is satisfied and we show below (see \eqref{conicalcond})  that there are none on $\partial B$. The lack of interior orbifold points can be explicitly seen by noting that in terms of the $(\tau,x,y,\phi)$ chart, the interior of $B$ is parameterized by $(x,y)$ lying in an open rectangle and the torus action simply acts as identifications on $(\tau,\phi)$.  Alternatively,  the conditions required in \cite[Thm 4]{Hollands:2008fm} are satisfied\footnote{To see this, one applies the results of \cite{Hollands:2008fm} (taking into account Remark 2 of Theorem 4)  to the stationary asymptotically Kaluza-Klein spacetime  $\mathbb{R} \times M$ with Lorentzian metric $g_5 = -\td t^2 + g$.}. 

On corner and interior points, $G_{AB}$ has rank 0 and rank 2 respectively.  On each boundary segment, $G_{AB}$ is rank 1, and admits a null vector $v^i \ell_i$ where $v^i \in \Z$ and $\ell_i$ is a basis of the Killing vector fields with $2\pi-$periodic orbits.  Hence the boundary segments represent fixed point sets of a particular generator of the torus action.  If $v^i, w^i$ are `rod vectors' corresponding to two adjacent boundary segments meeting at a corner,  then smoothness requires the `compatibility' condition
\begin{equation}\label{adjreg}
\begin{pmatrix} v^1 & w^1 \\ v^2 & w^2 \end{pmatrix}  \in GL(2, \Z). 
\end{equation} This condition is equivalent to requiring that $B$ has no orbifold singularities \cite{Hollands:2008fm} at the corner points.  $B$ is in fact homeomorphic to the upper half plane $\R^2_+$ \cite{Hollands:2007aj} .  The harmonic functions $(\rho,z)$ furnish global coordinates on  $B$ where $\rho > 0$ and $z \in \R$.  The boundary $\rho =0 $ corresponds to the fixed-point sets of the torus action.  This provides a convenient realization of the torus action on $M$ as follows.  The boundary segments are represented as intervals  $I_k : z_k < z < z_{k+1}$ on the $z-$axis with associated rod vectors $v^i_{(j)} \ell_i$. The corner points (or `nuts') correspond to points $z_k$ at which both $\ell_i$ vanish.  At interior points of $B$, $\rho > 0$, and the torus action is free.  We refer to the specification of the rods $(I_j, v^i_j)$  as the \emph{interval data}.  The interval data determines the topology of $M$.  

\subsubsection{Two-parameter family of gravitational instaontons on $\mathbb{CP}^2 \setminus S^1$}  Suppose that the quartic $P$ admits four real distinct roots $x_1 < x_2 < x_3 < x_4$. Note that this places restrictions on the parameters $a_i$.  To ensure that the family of cohomogeneity-two local metrics have positive definite signature, it is sufficient to restrict the range of $(x,y)$ to the rectangle 
\begin{equation}
-\infty < x_1 < y < x_2 < x < x_3 
\end{equation} with the asymptotic region corresponding to $x,  y \to x_2$. The sets $x = x_i, y = x_i$, $i = 1,2, 3$, correspond to fixed point sets of the torus action and the metric will degenerate there. 

Consider the two-parameter family $(\kappa, \xi)$ of metrics obtained by taking $P(x)$ to be the monic cubic polynomial with roots
\begin{equation}
x_1 = -4 \xi^3 (1 - \xi), \qquad x_2 = -\xi(1 - 2 \xi + 2 \xi^2) , \qquad x_3 = 1 - 2\xi.
\end{equation}
and a formal fourth root $x_4 = \infty$. Then
\begin{gather*}
\nu = -2\xi^2, \quad a_4 = 0, \quad a_3 = 1, \quad a_2 = -1 + 3 \xi - 2 \xi^2 + 6 \xi^3 - 4 \xi^4, \\ a_1 = -(\xi - 4 \xi^2 + 10 \xi^3 - 20 \xi^4 + 20 \xi^5 - 16 \xi^6 + 8 \xi^7), \quad a_0 =-4(1-\xi) \xi^4 (1 - 2 \xi)(1 - 2\xi + 2 \xi^2). 
\end{gather*} 

The $(x,y)$ coordinate chart degenerates in the asymptotic region $x, y \to x_2$.  Introduce a new chart $(r,\theta)$ implicitly by 
\begin{equation}\label{AFendchart}
x = x_2 - \frac{x_2 \sqrt{\kappa (1-\nu^2)}}{r} \cos^2 \frac{\theta}{2} , \qquad y = x_2 + \frac{x_2 \sqrt{k(1-\nu^2)}}{r} \sin^2 \frac{\theta}{2} . 
\end{equation} where $\theta \in (0,\pi)$. The asymptotic region corresponds to $r \to \infty$. A computation shows that
\begin{equation}
\frac{G}{F} \to  \frac{4 \xi^3 ( 1 - 4 \xi + 8 \xi^2 - 12 \xi^3 + 16 \xi^4 - 8 \xi^5)}{1 - 2\xi + 2 \xi^2} + O(r^{-1}), 
\end{equation} and
\begin{equation}
g_{rr} = 1 + \frac{\kappa (1 + 2\xi^2)^2}{\sqrt{\kappa(1 - 4 \xi^4)} r} + O(r^{-2}) , \qquad g_{\theta \theta} = r^2 ( 1 + O(r^{-1})) \qquad g_{\phi \phi} = O(r^2 \sin^2\theta)
\end{equation} This means that as $r \to \infty$ ,  the metric takes the manifestly asymptotically flat form
\begin{equation}\label{AFend}
g \to \td\hat \tau^2 + \td r^2 + r^2 (\td\theta^2 + \sin^2 \theta \td\hat \phi^2) + \mathcal{O}(r^{-1}) 
\end{equation} where $(\hat\tau, \hat \phi)$ are appropriately chosen linear combinations of  $(\tau, \phi)$.  Note that this is only possible because $G/F$ approaches a constant, and in particular there is no `NUT charge' (namely, there is no $N \cos\theta$ term at $O(1)$ order in the expansion). Regularity of the metric requires that appropriate identifications are made in the $(\hat \tau, \hat \phi)$ plane as described below. 

Note that the restriction $\nu^2 < 1$ requires that $|\xi| < 1/\sqrt{2}$. We also have to enforce the conditions $x_1 < x_2 < x_3$ . It can be checked these conditions are met if and only if $\xi \in (0, \frac{1}{2})$ or $\frac{1}{2} < \xi < 1/\sqrt{2}$.  However, the local metric has a curvature singularity when $y = -\nu x=2\xi^2 x$, where the function $H$ vanishes.  Since in addition $y < x_2 < x$, we can avoid $y = -\nu x$ if
\begin{equation}
y < x_2 < 2\xi^2 x_2 < 2\xi^2 x
\end{equation} which in turn requires $x_2 (1 - 2\xi^2) < 0$.  Since the second factor is positive, we require $x_2 = -\xi(1 - 2\xi + 2 \xi^2)< 0$.  The second factor is positive for all $\xi$,  so we require $\xi > 0$.  In summary we find one must restrict the parameters such that  $\kappa > 0$ and 
\begin{equation}
\frac{1}{2} < \xi < \frac{1}{\sqrt{2}}. 
\end{equation}

\subsubsection{The Rod Structure} To analyze the metric near the fixed points of the torus action, it is convenient to use the global coordinates on the orbit space $\rho \in \R^+, z \in \R$ introduced in \eqref{rho}, \eqref{z} respectively.  Under those coordinate transformation, the  vertices in the boundary rectangle in the $(x,y)$ plane are mapped to $(\rho,z) = (0,z_i)$ with  $z_1 < z_2 < z_3$, $z_i = x_i/2$, $i = 1,2,3$.  In detail, the vertices $(x,y) = (x_2,x_1) \to z_3, (x_3,x_1) \to z_2, (x_3,x_2) \to z_1$.  In the framework of the `interval structure' discussed above we have four rods: 
\begin{enumerate}
\item $I_1$: $z_3 < z < \infty$ or $x= x_2, y \in (x_1, x_2) $, a semi-infinite rod extending to $z \to \infty$
\item $I_2$: $z_2 < z < z_3$ or $y = x_1, x \in (x_3, x_3)$ , a finite rod;
\item $I_3$: $z_1 < z < z_2$ or $x = x_3, y \in (x_1, x_2)$ a finite rod;
\item $I_4$: $-\infty < z < z_1$ or $y = x_1, x \in (x_2, x_3)$, a semi infinite rod.
\end{enumerate} The associated normalized Killing fields $\ell_I$ that degenerate on each rod are normalized to generate $2\pi-$ periodic orbits.   On a given interval $I_I$, the restriction of the metric to the Killing vector fields is
\begin{equation}
\tilde{g}_{ij} =\frac{1}{H(x-y)}  \begin{pmatrix}  F & G \\ G & \frac{G^2}{F} \end{pmatrix}
\end{equation} which is obviously degenerate.  The associated normalized Killing fields $\ell_I$ that degenerate on each rod are normalized to generate $2\pi-$periodic orbits.  This requires 
\begin{equation}
\lim_{\rho \to 0} \frac{ \td|\ell_I|^2 \cdot \td|\ell_I|^2}{4 |\ell_I|^2}\Bigg\vert_{z \in I_I} =1.
\end{equation}  Explicitly, in terms of the generators $(\partial_\tau, \partial_\phi)$, a computation gives
\begin{equation}\label{ellinspi}
\ell_I = \frac{1}{k_I} \left( b_I \frac{\partial}{\partial \tau} + \frac{\partial}{\partial \phi} \right)
\end{equation} where
\begin{equation}
\begin{aligned}
k_1  &= \frac{(1 - \xi) ( 1 - 2 \xi) ( 1 - 2\xi^2)^2}{2 \sqrt{\kappa} (1 - 2\xi + 2 \xi^2)},  \qquad  k_2 = \frac{(1 - 2\xi)(1 - 2 \xi^2)^2 (1 - 2\xi + 2 \xi^2)}{8 \sqrt{\kappa} (1- \xi) \xi^2}, \\
k_3 & = \frac{(1 - \xi)(1 - 2\xi^2)^2(1 - 2\xi + 2\xi^2)}{2 \sqrt{\kappa} ( 1 - 2\xi)} , \qquad k_4 = k_1,
\end{aligned} 
\end{equation} and 
\begin{equation}
\begin{aligned}
b_1   &= \frac{4 \xi^3 (-1 + 4\xi (1 - \xi)^2(1 + 2 \xi^2))}{1 - 2\xi (1 - \xi)} , \qquad b_2  = \frac{\xi^2 ( 1 - 2\xi(2 - 3\xi + 10 \xi^2 - 16 \xi^3 + 8 \xi^4))}{1 - \xi}, \\
b_3 & =  \frac{4\xi^3(1-3\xi +7 \xi^2-12\xi^3+6\xi^4)}{(2 \xi -1)}, \qquad b_4 = b_1.
\end{aligned}
\end{equation}  

 It is easily verified that
\begin{equation}\label{conicalcond}
\ell_1 = \ell_4 = \ell_2 + \ell_3.
\end{equation} and so in terms of the basis of $T^2$ generated by $(\ell_1, \ell_2)$, the rod vectors $v_I^i$ are given by $v_1^i = (1,0), v_2^i = (0,1), v_3^i = (-1,1)$, and $v_4^i = (1,0)$.  It is easy to observe that the regularity condition on adjacent rod vectors \eqref{adjreg} are satisfied. 

Let us introduce angular coordinates $\phi^i, i = 1,2$ adapted to  the Killing vector fields $\ell_i$, so that $\phi^i \sim \phi^i + 2\pi$ and 
\begin{equation}
\frac{\partial}{\partial \phi^i } = \ell_i. 
\end{equation} In terms of the original angles $(\tau, \phi)$ we have
\begin{equation}
\tau = \frac{b_1}{k_1} \phi^1 +  \frac{b_2}{k_2} \phi^2, \qquad \phi = \frac{\phi^1}{k_1}  + \frac{\phi^2}{k_2} ,
\end{equation} The identifications 
\begin{equation}
(\phi^1, \phi^2) \sim (\phi^1 + 2\pi, \phi^2), \qquad (\phi^1, \phi^2) \sim (\phi^1, \phi^2 + 2\pi),
\end{equation} induce the following identifications in $(\tau, \phi)$ plane: 
\begin{equation}
(\tau, \phi) \sim \left(\tau + \frac{2\pi b_1}{k_1} , \phi + \frac{2\pi}{k_1} \right),  \qquad  (\tau, \phi) \sim \left( \tau + \frac{2\pi b_2}{k_2}, \phi + \frac{2\pi}{k_2}\right).
\end{equation} Note that the coordinates $(\hat\tau, \hat\phi)$ introduced in the asymptotic region (see \eqref{AFendchart}) are given by
\begin{equation}
\hat\phi = \phi^1 + \frac{4 ( 1- \xi)^2 \xi^2}{(1 - 2\xi + 2 \xi^2)^2} \phi^2, \qquad  \hat\tau = \frac{1}{\sqrt{1 - \nu^2}}  \frac{8 \sqrt{\kappa} \xi^4}{(1 - 2\xi + 2\xi^2)^2} \phi^2
\end{equation} in terms of which the above identifications can be written 
\begin{equation}
(\hat \tau, \hat \phi) \sim (\hat \tau , \hat \phi + 2\pi), \qquad (\hat\tau, \hat \phi) \sim \left(\hat \tau +  \frac{16\pi}{\sqrt{1 - \nu^2}}  \frac{ \sqrt{\kappa} \xi^4}{(1 - 2\xi + 2\xi^2)^2}, 
 \hat\phi + \frac{8\pi ( 1- \xi)^2 \xi^2}{(1 - 2\xi + 2 \xi^2)^2}  \right) . 
\end{equation} This demonstrates that the metric $g$ is indeed asymptotically flat in the sense discussed in \S \ref{ALFintro} for appropriate choices of $(\beta, \Omega)$. 

Let us consider the requirement that the Killing field with constant norm at infinity actually generates closed orbits (i.e. a $U(1)$-action rather than an $\R$-action). In terms of the $(\ell_1, \ell_2)$ basis, 
\begin{equation}
\frac{\partial}{\partial \tau} = \frac{k_2}{(b_1 - b_2)} \left [\frac{k_1}{k_2} \ell_1 - \ell_2 \right].
\end{equation}  Therefore $\frac{\partial}{\partial \tau}$ generates a $U(1)$-action if and only if 
\begin{equation}
\frac{k_1}{k_2} \in \mathbb{\mathbb{Q}}.
\end{equation} 
This $U(1)$ action would restrict to a $U(1)$-principal bundle near infinity only if  $\frac{k_1}{k_2} \in \Z$. However it can be verified that this integrality condition is not satisfied for any allowed value of $\xi$.  

We conclude this section by verifying that we have constructed harmonic forms with finite $L^2$ energy. We calculate the energy explicitly in (\ref{intersection theory}).  

\begin{prop}\label{finen} The (anti)-self-dual forms $\omega_\pm$ generated by the smooth functions \eqref{alphapm}, and the anti-self-dual form $\omega_2$ generated by \eqref{alpha2} have finite $L^2$ energy.
\end{prop}
\begin{proof} These harmonic forms have $L^2$ energy given by integrating the energy density $E[\alpha] \in \Lambda^4 M$  \eqref{L2cond} over $M$. Since the $\alpha$ are smooth these integrals will be finite over any compact $K \subset M$. Outside of this compact set we may introduce the asymptotically flat chart $(r,\theta)$ defined by \eqref{AFend}  . A computation shows that as $r \to \infty$, 
\begin{equation} \begin{aligned}
E[\alpha_+] & = E[\alpha_-]  =\left( \frac{\kappa (1 - \nu^2)^2}{(1 - 2 \xi^2)^4 r^4} + O(r^{-5})\right) \td \text{Vol}_g, \\ E[\alpha_2 ] &=\left( \frac{\kappa}{(1 - \xi)^2 \xi^2 (1 - 2\xi)^2 (1 - 4 \xi^2)^2 r^4} + O(r^{-5})\right) \td \text{Vol}_g, 
\end{aligned}
\end{equation}  where $\td \text{Vol}(g) \sim \mathcal{O}(r^2) \sin\theta \td \tau \wedge \td r \wedge \td \theta \wedge \td \phi$ as $ r \to \infty$.  Clearly there is a finite contribution to the $L^2$ energy for $R < r < \infty$ for some fixed $R$ and so outside a compact set, the $L^2$ energy is finite for this set of harmonic forms.  Hence the total $L^2$ energy is finite.   
\end{proof}\noindent We compute the $L^2$ energy exactly in \S \ref{INMPF}.

\section{Period integrals}\label{Period integrals}
The harmonic forms can be integrated over the bolts using the Atiyah-Bott fixed point localization formula (see \cite{AB} (3.8)).  In the present case the formula can be derived using elementary methods as follows.

Suppose $N$ is diffeomorphic to $S^2$ and  $X$ is a vector field on $N$ that integrates to a periodic flow of period $T > 0$. Then there exist cylindrical coordinates $\phi \in [0, 2\pi)$, $z \in [-1,1]$ on $N$ in which $X = \frac{2\pi}{T} \frac{\partial}{\partial \phi}$. Denote by $n,s \in N$ the north and south poles at $z=1,-1$ respectively.

Now suppose $\omega \in \Omega^2(N)$ is a 2-form for which $L_X \omega = 0$.  Then $\omega = f(z) \td \phi \wedge \td z$ for some smooth function $f: N \rightarrow \R$ depending only on $z$. Integrating gives
\begin{equation}\label{locform}
\pm \int_{N} \omega  =  2 \pi \int_{-1}^1 f(z) \td z  = T ( g(n) - g(s))
\end{equation}
where $g$ is a function satisfying $\td g =  \iota_X \omega =  \frac{2\pi}{T}   f(z) \td z$. The overall sign depends on the choice of orientation for $N$.

We can use (\ref{locform}) to integrate 2-forms
\begin{equation}
 \omega = \frac{1}{|K|^2} \left[ K \wedge \td \alpha - \star (K \wedge \td \alpha)\right]
 \end{equation}
over the finite bolts $B_I$, $I=2,3$.  This is valid because the vector field $K = \frac{ \partial}{\partial \tau}$ is tangent to $B_I$ and integrates to a periodic action on $B_I$. Since $ \iota_K \omega = \td \alpha$ we have 
\begin{equation}\label{finiterodint}
\pm \int_{B_I}  \omega  =  T ( \alpha(n) - \alpha(s))
\end{equation}
where $n,s$ are the two nuts of $B_I$ and  $L$ is the period of $K$ restricted to $B_I$.   Similarly, for the half infinite bolts $B_1$ and $B_4$, if let $s$ be the one nut, then
\begin{equation}\label{infiniterodint}
\pm \int_{B_I}  \omega  =  T ( \alpha(\infty) - \alpha(s))
\end{equation}
where $\alpha(\infty)$ is the limit of $\alpha(p)$ as $p$ goes to infinity along the bolt. 

To determine the period $T$, express 
$$ \frac{ \partial}{\partial \tau}  =  c_1 \ell_1  + c_I \ell_I.$$
Then since $\ell_I$ vanishes on $B_I$ and $\ell_1$ has period $2\pi$,  we see $K = \frac{ \partial}{\partial \tau}$ has period $2 \pi /|c_1|$.  In terms of formula (\ref{ellinspi}), for $I =2,3$ we have  
\begin{equation}\label{finrodintct}
\pm  \frac{1}{2\pi} \int_{B_I}  \omega =  \frac{(b_1 - b_I)}{k_1} ( \alpha(n) - \alpha(s)).
 \end{equation}

Similarly, for the infinite bolts $I=1,4$ we have
 \begin{equation}
\pm  \frac{1}{2\pi} \int_{B_I}  \omega =   \frac{(b_2 - b_I)}{k_2}  ( \alpha(\infty) - \alpha(s)).
 \end{equation}

\subsection{The anti-self-dual instantons}

The homology group $H^2(M;\Z) \cong \Z^2$ is generated by the two finite bolts $B_2$ and $B_3$.  Therefore a harmonic 2-form in $\mathcal{H}^2_-(M,g)$ represents an instanton if and only if its integrals over $B_2$ and $B_3$ both lie in $2\pi \Z$.

Recall that on $B_2$, $y = x_1$ and $ x_2 < x< x_3$ and on $B_3$, $x = x_3, x_1 < y < x_2$.  Applying (\ref{finrodintct}) we find
\begin{equation}
 \frac{1}{2\pi} \int_{B_2} \omega_{-}  = \frac{(b_1 - b_2)}{k_1} (\alpha_-(x_2,x_1) - \alpha_-(x_3,x_1) ) = \frac{2 \sqrt{\kappa}}{1 - 2\xi^2}.
\end{equation} Similarly we find
\begin{equation}
\frac{1}{2\pi} \int_{B_3} \omega_{-}  = \frac{(b_1 - b_3)}{k_1} \left(\alpha_-(x_3,x_1) - \alpha_-(x_3, x_2) \right)  = \frac{4 \xi^2 \sqrt{\kappa} }{1 - 2\xi^2}.
\end{equation}
Notice that by dropping the sign ambiguity, we have implicitly chosen orientations for $B_2$ and $B_3$. 

Since $\int_{B_I} \td K=0$ for $I=2,3$ by Stokes' Theorem,
\begin{equation}
 \frac{1}{2\pi} \int_{B_I} \omega_{+}   = \frac{1}{2\pi} \int_{B_I} \star \td K =  - \frac{1}{2\pi} \int_{B_I} \omega_{-}
\end{equation}
so therefore we have a ratio of periods
\begin{equation}\label{eqna}
 \frac{ \int_{B_2} \omega_{\pm}}{  \int_{B_3} \omega_{\pm}} =  \frac{1}{2\xi^2} .
 \end{equation}

Applying (\ref{finrodintct}) to the second anti-self-dual harmonic form $\omega_2$ associated to our second solution $\alpha_2$, we find
\begin{equation}
 \frac{1}{2\pi} \int_{B_2} \omega_2 =\frac{2 \sqrt{\kappa }}{\xi 
   (2 \xi-1 ) \left(1 -2\xi + 2 \xi ^2 \right) \left(1-4 \xi^4\right)}
\end{equation}  and 
\begin{equation}
\frac{1}{2\pi}  \int_{B_3} \omega_2 = \frac{2 \sqrt{\kappa }}{(1 -\xi ) 
  \left(1 -2\xi + 2 \xi ^2 \right)\left(1- 4 \xi^4 \right)}.
\end{equation}  
and therefore we have a ratio
\begin{equation}\label{eqnb}
 \frac{ \int_{B_2} \omega_2}{  \int_{B_3} \omega_2 }=  \frac{(\xi-1)}{\xi(1-2\x)}  .
 \end{equation}
The ratios (\ref{eqna}),(\ref{eqnb}) are unequal for the allowed values of $\xi \in ( \frac{1}{2}, \frac{1}{\sqrt{2}})$. Since $B_2, B_3$ form a basis for $H_2(M;\Z)$ we conclude that $\omega_-$ and $\omega_2$ represent linearly independent elements of $ H^2(M;\R)$ implying that the map $j$ in (\ref{commdiag}) is surjective.

To get an explicit formula for $F_A$, let
\begin{align}
\tilde{\omega}_- & := \frac{-1}{2 \sqrt{\kappa} } \omega_- &  \tilde{\omega}_2 & :=   \frac{\xi (2\xi-1)(1-\xi)(1-2\xi+2\xi^2)(2\xi^2+1)}{2 \sqrt{\kappa} } \omega_2
\end{align}
then
\begin{align}
 \nu_2 & := -  \tilde{\omega}_- + 2 \xi (\tilde{\omega}_- + \tilde{\omega}_2)  & \nu_3 & :=  \tilde{\omega}_-  - \xi^{-1}(\tilde{\omega}_- + \tilde{\omega}_2)
\end{align}
is the dual basis to $B_2, B_3$ satisfying for $I,J \in \{2,3\}$
$$ \frac{1}{2\pi} \int_{B_I} \nu_J  = \delta^I_J.$$ 

\begin{prop}
Let $P$ be a $U(1)$-principal bundle over $M$ with first Chern class $c_1(P) \in H^2(M;\Z)$ satisfying
$$ \int_{B_I} c_1(P) = m_I $$
for $I=2,3$.  Then $P$ admits a unique anti-self-dual instanton $A$ (up to gauge equivalence). It has curvature  
$$F_A  = m_2 \nu_2 + m_3 \nu_3.$$ 
\end{prop}

\begin{proof}
We have proven that $\mathcal{H}_-(M,g)$ is two dimensional and that the natural map $j: \mathcal{H}(M,g) \rightarrow H^2(M;\R)$ restricts to an isomorphism $\mathcal{H}_-(M,g) \cong H^2(M;\R)$ meaning that each de Rham cohomology class contains a unique anti-self-dual $L^2$ harmonic representative.  Since $H^2(M;\Z)$ is torsion free, it embeds into $H^2(M;\R)$, so the Chern class $c_1(P)$ also admits a unique anti-self-dual $L^2$ harmonic representative which much equal $ \frac{1}{2\pi} F_A$ for an anti-self-dual connection $A$ on $P$.  Since $\pi_1(M) =0$, there are no non-trivial flat connections over $M$, and we deduce that this connection $A$ is unique up to gauge equivalence. 
\end{proof}

\begin{remark}
Explicit gauge potentials for the instantons can be produced using (\ref{Apm}) and (\ref{A2}).
\end{remark}

\subsection{The compactly supported image}
In the compactification $\overline{M} = \C P^2$, the cycles $B_2$, $B_3$, and $\overline{B_1 \cup B_4}$ all represent the same homology class up to orientation; it can be verified that with our chosen orientations $[B_2] = -[B_3] = -[\overline{B_1 \cup B_4}]$. Therefore the image of $i$ in (\ref{commdiag}) is spanned by $dK$ and $\nu_2-\nu_3$, since these are linearly independent, finite energy harmonic forms for which $\int_{B_2} \omega =- \int_{B_3} \omega$.  In particular

\begin{equation}\label{stokesapp}
\int_{B_2} \td K =- \int_{B_3} \td K =0
\end{equation}
\begin{equation}
\frac{1}{2\pi} \int_{B_2} \nu_2-\nu_3 =- \frac{1}{2\pi} \int_{B_3} \nu_2-\nu_3 = -1
\end{equation}
Integrating along the infinite bolts using (\ref{infiniterodint}) (for appropriate choice of orientation), we get
\begin{equation}
\frac{1}{2\pi}  \int_{B_1} \nu_2 -\nu_3 =  \frac{2 \xi^2}{2\xi^2+1} 
 \end{equation}
\begin{equation}
\frac{1}{2\pi}  \int_{B_4} \nu_2 -\nu_3 =  \frac{1}{2\xi^2+1}
 \end{equation}
so that $\frac{1}{2\pi}  \int_{B_1 \cup B_4} \nu_2 -\nu_3 = 1 $. 

To evaluate $\frac{1}{2\pi}  \int_{B_I} \td K$, first observe that  $\iota_K \td K = L_K K - \td \iota_K K = -\td |K|^2$.  Since $|K|=0$ at the nuts, formula (\ref{finiterodint}) confirms (\ref{stokesapp}), and yields the formula

\begin{equation}
\frac{1}{2\pi}   \int_{B_1} \td K =   - \frac{1}{2\pi}  \int_{B_4} \td K =   \frac{b_2-b_1}{k_2} |K|^2(\infty) =  \frac{-8 \xi^4 \sqrt{\kappa}}{ (1-2\xi^2)(2\xi^2+1) (2\xi^2-2\xi+1)^2 }
 \end{equation}
where $|K|^2(\infty)$ is the limiting value of $|K|^2$ at infinity (see \ref{K2inf}).

Therefore, if we define 
\begin{align} 
\mu_1 &:=  \frac{ (1-2\xi^2)(2\xi^2+1) (2\xi^2-2\xi+1)^2}{-8 \xi^4 \sqrt{\kappa}} dK & \mu_2 &:= \nu_3-\nu_2 + \frac{2\xi^2}{2\xi^2+1} \mu_1
\end{align} then  $\mu_1, \mu_2$ lie in the image of $i$ and 
$$\frac{1}{2\pi} \int_{B_I} \mu_J = \delta^I_J $$
for $I,J \in \{1,2\}$.

\section{Intersection numbers and the Maxwell partition function}\label{INMPF}

In Maxwell theory with theta term \cite{W, T}, we have an action
\begin{equation}
S(A)  :=  \frac{1}{g^2} \int_M  F_A \wedge \star F_A  + i \frac{ \theta}{8\pi^2}\int_M  F_A \wedge F_A 
\end{equation}
which depends on the complex paramater $\tau :=\frac{\theta}{2\pi}  +  i \frac{4\pi}{g^2} $.

The theta term is purely topological.  If we express 
\begin{equation}
F_A  = \td a + \sum_{i=1,2} m_i \nu_i
\end{equation}
then, assuming appropriate asymptotics of $da$, by Stokes' Theorem
\begin{equation}
\int F_A \wedge F_A =  \int_M \td a \wedge \td a +  \sum_{i=1,2} 2m_i \int_M \nu_i \wedge \td a  + \int_M \nu_i \wedge \nu_j  = 4 \pi^2 \vec{m} Q \vec{m}^T
\end{equation}
where $4 \pi^2 Q_{i,j} = \int_M \nu_i \wedge \nu_j$ . Similarly
\begin{equation}
\int F_A \wedge \star F_A= \int_M \td a \wedge \star \td a - 4 \pi^2 \vec{m} Q \vec{m}^T
\end{equation} 
exploiting that the $\nu_i$ are anti-self-dual. Notice this implies that the anti-self-dual connection is the unique global minimum for $Re(S (A))$ for each $U(1)$-bundle $P$. The partition function factors
\begin{equation}
  Z(\tau) =  \int_{A}  e^{-S(A)} \td A = Z_c(\tau) Z_q(\tau)
  \end{equation}
where
\begin{equation}
 Z_c(\tau) := \sum_{\vec{m}  \in \Z^2} \exp\left( \frac{4\pi^2}{g^2} \vec{m} Q \vec{m}^T  - i \frac{ \theta}{2\pi} \vec{m} Q \vec{m}^T\right)  = \sum_{\vec{m}  \in \Z^2} \exp \left( -i \pi  \vec{m} Q \vec{m}^T \tau \right)
 \end{equation}
and
\begin{equation}
 Z_q(\tau)  := \int   Da \exp \left(  \frac{-1}{g^2}  \int_M \td a \wedge \star \td a \right) =   \int Da  \exp   \left( -Im(\tau)    \frac{1}{ 4 \pi}   \int_M \td a \wedge \star \td a  \right) .
 \end{equation}
To determine $Z_c(\tau)$ it only remains to calculate the intersection matrix $Q$. We carry this out in the next section.

The $ Z_q(\tau)$ factor should be understood as a regularized determinant. By analogy with (\cite{W} (2.6)) we expect $Z_q(\tau)$ to equal a constant times $Im(\tau)^{k/2}$ for some integer $k$, but we do not carry out the analysis in this paper.

\subsection{The intersection pairing}\label{intersection theory}

It remains to understand the ``intersection pairing"

\begin{align}
\mathcal{H}^2(M,g) \times \mathcal{H}^2(M,g) &\rightarrow \R, &  (\alpha,\beta) & \mapsto  \int_M \alpha \wedge \beta.
\end{align}

From topology, we have natural pairing  
\begin{equation}\label{natpa}
 H^2_c(M) \otimes H^2(M) \rightarrow H^4_c(M) \cong \R
 \end{equation}
which, at the level of differential forms, sends $( \omega_1, \omega_2) \mapsto  \int_{M} \omega_1 \wedge \omega_2$.  This is well defined because the product of a compactly supported 2-form with a 2-form will be compactly supported 4-form, and hence will be integrable and this descends to a pairing on cohomology classes by standard arguments using Stokes' theorem.  When restricted to integral cohomology, (\ref{natpa}) defines a unimodular matrix due to Poincar\'e duality.

Recall our morphisms (\ref{commdiag})
$$ H_c^2(M) \stackrel{i}{\rightarrow} \mathcal{H}^2(M,g) \stackrel{j}{\rightarrow} H^2(M).$$

\begin{lemma}
Given $\alpha \in H^2_c(M)$ and $\omega \in \mathcal{H}^2(M,g)$ we have the identity 
\begin{equation}
\int_M i(\alpha) \wedge \omega  =  \int_M \alpha \cup j(\omega).
\end{equation}
\end{lemma}

\begin{proof}
As explained in \cite{SS} if $\tilde{\alpha}$ is a compactly supported differential form, then  $i(\alpha) = \tilde{\alpha} - \beta$ where $\beta = \lim \td\gamma_i$ in L2, where the $\gamma_i $ are compactly supported one forms.  Therefore $ \int_M \beta \wedge \omega  = lim \int_M \td \gamma_i \wedge \omega = 0$ by Stokes' Theorem.
\end{proof}

This Lemma combined with the invariance of the pairing under $\star$ suffices to determine the pairing completely. For $ \alpha, \beta \in \mathcal{H}(M,g)$, denote 
\begin{equation}
\langle \alpha , \beta \rangle :=  \frac{1}{4\pi^2}  \int_M \alpha \wedge \beta.
\end{equation} 

In terms of the integral basis, the $2 \times 2$ matrix 
\begin{equation}
B_{I,J} =  \langle  \mu_I, \nu_J\rangle
\end{equation} 
must have integer entries and have determinant $\pm 1$. 

\begin{lemma}
We have $B_{12} = B_{13} = - B_{23} =1$ and  $B_{22} = 0$. 
\end{lemma}

\begin{proof}

Pairings between forms in the image of $i: H^2_c(M)  \rightarrow \mathcal{H}^2(M,g) $ can be understood using the commuting diagram 
\begin{equation}
 \xymatrix{ H^2_c(M) \otimes H^2_c(M) \ar[r]  \ar[d] & H^4_c(M) \ar[d]^{\cong} \\  H^2(\C P^2) \otimes H^2(\C P^2)  \ar[r] & H^4(\C P^2)}.
 \end{equation}
Since $\mu_1$ is exact and both $\mu_1$ and $\nu_2-\nu_3$ lies in the image of $i$ we know that 
\begin{equation}  \langle \mu_1 ,\mu_1 \rangle  =  \langle \mu_1, \nu_2 -\nu_3\rangle  =0
\end{equation}
so unimodularity forces
\begin{equation}
 \langle   \mu_1 , \nu_2 \rangle  =  \langle  \mu_1 ,  \nu_3 \rangle  = \pm 1 . 
 \end{equation}

The sign agrees with  
\begin{equation}
\langle   \mu_1 ,  \xi \nu_3 + \frac{1}{2\xi} \nu_2 \rangle  = \langle   \mu_1 ,  (\xi -\frac{1}{2\xi}) \tilde{\omega}_- \rangle = -C (\xi -\frac{1}{2\xi})  \langle \td K, \star \td K\rangle 
\end{equation}
which is positive.

By unimodularity and the fact that $\nu_3-\nu_2$ is anti-self-dual it follows that
\begin{equation}\label{idcom}
 \langle  \mu_2 , \mu_2 \rangle = \langle  \mu_2 ,\nu_3-\nu_2 \rangle =  \langle  \nu_3-\nu_2, \nu_3-\nu_2 \rangle = -1.
 \end{equation}

Next we use Poincare duality (we'll be sloppy with signs since orientations won't matter). Let $PD(B_1)$ denote the Poincare dual of $B_1$. Then since $B_1$ intersects $B_2$ transversely at one point and intersects $B_3$ trivially, we see that $ PD(B_1)  = \pm j( \nu_2)$. By similar reasoning  $i( PD(B_3) ) = \pm \mu_2$.  Therefore
\begin{equation}
 \langle  \mu_2 , \nu_2 \rangle = \frac{1}{4\pi^2}\int_M PD(B_3) \cup PD(B_1) =0
 \end{equation}
since $B_1$ and $B_3$ do not intersect. From (\ref{idcom}) we deduce that
\begin{equation}
  \langle \mu_2 , \nu_3  \rangle = -1.
  \end{equation}
\end{proof}

Next we write 
\begin{equation}
 \tilde{\omega}_-  =  \frac{-1}{2 \sqrt{\kappa} } (\td K - \star \td K) =  \frac{1}{2 \xi^2-1} (\nu_2 +2\xi^2 \nu_3)
 \end{equation}

By anti-self-duality we have $  \langle \td K , \nu_I \rangle  = - \langle \star \td K , \nu_I \rangle$ so if $\mu_1 = \frac{C}{\sqrt{\kappa}} \td K$ then
\begin{equation}
    \frac{1}{2 \xi^2-1}  \langle \nu_2 +2\xi^2 \nu_3,  \nu_I \rangle =  \frac{-1}{C} \langle  \mu_1 , \nu_I \rangle  =   -  \frac{1}{C   }  =   \frac{-8 \xi^4}{ (2\xi^2-1)(2\xi^2+1) (2\xi^2-2\xi+1)^2 }.
    \end{equation}
Let
\begin{equation}
  A  :=   \langle \nu_2 +2\xi^2 \nu_3 , \nu_I \rangle =    \frac{-8 \xi^4 }{ (2\xi^2+1) (2 \xi^2-2\xi+1)^2 }.
  \end{equation}
Letting $Q_{IJ} = \langle \nu_I , \nu_J \rangle$ we obtain a linear system of equations

\begin{equation}
 \begin{bmatrix}  1 & -2 & 1\\  1 & 2\xi^2 & 0 \\ 0 & 1 & 2\xi^2  \end{bmatrix} \begin{bmatrix}  Q_{22}\\  Q_{23} \\ Q_{33}  \end{bmatrix} = \begin{bmatrix}  1\\  A \\ A  \end{bmatrix}
 \end{equation}

Solving gives

\begin{eqnarray}
Q_{33} &=&  B \\
Q_{23}= Q_{32} &=&  A-2\xi B \\ \label{system}
Q_{22} &=&  1+2A- (1+4\xi)B 
\end{eqnarray}
where 
\begin{equation}
B := \frac{(3 +2\xi)A-1}{4\xi^2+4\xi+1}.
\end{equation}

Notice that the intersection pairings $Q_{IJ}$ are not integers. This can be contrasted with the case of compact manifolds considered in \cite{W} \cite{T} where Poincar\'e duality forces the $Q_{IJ}$ to be integers. In particular $Z_c(\tau)$ is not a modular form for the Chen-Teo gravitational instanton.

\begin{remark}
In \cite{EN}, the modularity property of the the Maxwell partition function is extended the certain ALF metrics which are called ``almost compact" by the authors, including in particular the Euclidean Schwartzschild and Euclidean Taub-NUT metrics (they erroneously include Euclidean Kerr which is not ALF). The essential idea is to include only connections that have trivial holonomy at infinity. In particular, one imposes the Dirac quantization condition on the semi-infinite bolts in addition to the finite bolts.  This strategy does not work in our situation, because it would impose four linearly independent conditions on the three dimensional vector space $\mathcal{H}^2(M,g)$.
\end{remark}

\begin{remark}
There is an alternative calculation of $Q$ using Stoke's Theorem. The pairings between $\td K, \star \td K$, and $\nu_2-\nu_3$ are easily deduced except for $\langle  \td K , \star \td K \rangle $.  
The boundary at infinity is a trivializable fibre bundle  $\partial M = S^2 \times S^1 \cong B_{\infty} \times O_{\infty}$ where $B_\infty$ is any fibre of the bundle and $O_{\infty} $ is the limiting $K$ orbit along the infinite bolt $B_1$.  Applying Stokes theorem, Fubini's Theorem, and the equality $ [B_\infty] = [B_2] + [B_3]$ in $H_2(M;\Z)$
\begin{gather}
 \int_M \td K \wedge \star \td K =  \left(\int_{S^1} K \right) \left(\int_{B_\infty} \star \td K \right) =  \left(\int_{B_1} \td K \right) \left(\int_{B_2}  \star \td K + \int_{B_3} \star \td K \right)\\
 = 4\pi^2 \left( \frac{8 \xi^4 \sqrt{\kappa}}{ (1- 2\xi^2)(2\xi^2+1)(2\xi^2-2\xi+1)^2 }\right) \left(  \frac{2 \sqrt{\kappa}}{1 - 2\xi^2} +  \frac{4 \xi^2 \sqrt{\kappa} }{1 - 2\xi^2} \right)\\
  =  \frac{64\pi^2 \xi^4 \kappa}{ (1-2\xi^2)^2 (2\xi^2-2\xi+1)^2 }.
 \end{gather}

\end{remark}

\appendix

\section{Asymptotically Flat Gravitational Instantons}\label{ALFintro}

Consider the flat Riemannian manifold $(M_{\flat}, g_{0})$ defined as the quotient space $M_{\flat}  := \R^4/\Z$ of Euclidean $\R^4 = \R \times \R^3$ under the automorphism $(\tau, r, ,\theta, \phi) \mapsto (\tau + \beta, r,  \theta, \phi+ \beta\Omega)$ , where $\beta, \Omega \in \R$ are certain constants, $\tau \in \R$ parameterizes the first $\R$ factor,  and $(r, \theta ,\phi)$ are spherical coordinates on $\R^3$ with $r>0, \theta \in (0, \pi), \phi \sim \phi + 2\pi$. In this chart we may express the flat metric as 
\begin{equation}\label{model}
g_0 = \td \tau^2 + \td r^2 + r^2 (\td \theta^2 + \sin^2\theta \td \phi^2).
\end{equation}
Both $\frac{\partial}{\partial \tau}$ and $\frac{\partial}{\partial \phi}$ are Killing vector fields on $M_{\flat}$, and $\frac{\partial}{\partial \tau}$ has bounded norm. We can think of $\frac{\partial}{\partial \tau}$ and $\frac{\partial}{\partial \phi}$ as Euclidean analogues of time translation and rotation respectively.

An \emph{asymptotically flat gravitational instanton} (or AF instanton) is a Riemannian manifold $(M,g)$  which is  geodesically complete, Ricci flat, and approaches $(M_\flat, g_0)$ asymptotically at infinity for some choice of $\beta$ and $\Omega$. We will not require precise decay rates and refer the reader to the discussion given in~\cite{L}.  Note that AF gravitational instantons are necessarily non-compact. 

In the above definition, there is a preferred vector field $K = \frac{\partial}{\partial \tau}$ with constant norm in the asymptotic region. We emphasize that the orbits of $K$ \emph{need not be closed} (that is, $\tau$ need not parameterize an $S^1$). The classic example of an AF gravitational instanton is the two-parameter family of Ricci flat Euclidean Kerr metrics on $\R^2 \times S^2$, 
\begin{equation}\label{EKerr}
g_K  = \frac{f}{\rho^2}\left(\td \tau + a \sin^2\theta \td \phi\right)^2 + \rho^2\left(\frac{\td r^2}{f} + \td \theta^2\right) + \frac{\sin^2\theta}{\rho^2} \left((r^2 - a^2) \td \phi- a \td \tau\right)^2
\end{equation}  where $f = r^2 - 2m r - a^2$, and $\rho^2 = r^2 - a^2 \cos^2\theta$.  The family is parameterized by constants $m>0, a \geq 0$. The radial coordinate $ r \in [r_+, \infty)$ where $r_+ = m + \sqrt{m^2 + a^2}$ is the positive root of $f$ and $\theta \in (0,\pi)$. The manifold admits a torus action as isometries generated by $\frac{\partial}{\partial \tau}, \frac{\partial}{\partial \phi}$ and the $(\tau,\phi)$ plane is identified as $(\tau, \phi) \sim (\tau, \phi+ 2\pi)$ and  $(\tau,\phi) \sim (\tau + \beta , \phi+ \beta \Omega)$ for certain constants $\beta, \Omega$ which depend on $m, a$. The metric is indeed geodesically complete and as $r \to \infty $ approaches the model metric \eqref{model} and hence is an AF instanton.

\begin{remark}
Thus this notion of AF instanton overlaps with, but does not contain, the related class of Ricci flat, \emph{asymptotically locally flat} (ALF) manifolds (see, e.g. \cite{Minerbe}).  
However,  it is not ALF because the vector field $K = \frac{\partial}{\partial \tau}$, which satisfies $|K| \to 1$ as $r \to \infty$,  does not have closed orbits.  Rather, the vector field $\hat{K} = \frac{\partial}{\partial \tau } - \frac{a}{r_+^2-a^2} \frac{\partial}{\partial \phi} $ which does indeed generate closed orbits and generate an $S^1$ in the asymptotic region $r \to \infty$, grows linearly in $r$ . Note that the ALF Euclidean Schwarzschild instanton is recovered from \eqref{EKerr} upon setting $a =0$. 
\end{remark}

The above definition of asymptotic flatness was guided by the definition of asymptotic flatness in Lorentzian manifolds that are asymptotic to  Minkowski spacetime (see the review  \cite{Chrusciel:2010fn} for precise definitions).  Indeed the Euclidean Kerr and Schwarzschild gravitational instantons can easily be obtained from their associated Lorentz signature, asymptotically Minkowskian Ricci-flat black hole metrics  by `analytic continuation' of the time coordinate $t \to i \tau$ of their associated Lorentz-signature, asymptotically flat, Ricci-flat black hole metrics along with suitable continuation of parameters.   The fact that the Kerr spacetime exhausts the set of asymptotically flat black hole solutions motivated the conjecture, discussed above, that Euclidean Kerr would exhaust the set of AF gravitational instantons admitting a torus action (assuming that the vector field  $K$ which has constant norm at infinity generates an isometry) \cite{L}.  In fact one can establish this result in the special case that the Riemannian instanton is `static' (i.e. $K$ is orthogonal to a family of hypersurfaces) by modifying the static black hole uniqueness theorem of Israel~\cite{Israel}.  However the standard proof of black hole uniqueness for stationary, axisymmetric solutions does not carry over to the Riemannian setting.  Indeed, these proofs make use of the fact that the stationary, axisymmetric Ricci-flat equations with Lorentizan signature reduce to a harmonic map with negatively-curved Riemannian target space; the analogous formulation in the Riemannian case yields a target space with Lorentzian signature.

\section{The reduction to three dimensions} \label{3Dr}

Let us define two scalar potentials, called the Ernst potentials
\begin{equation}
\mathcal{E}_ + = \frac{1}{2} (\lambda +  \tau)  \qquad \mathcal{E}_- = \frac{1}{2} (\lambda - \tau) 
\end{equation}
If one uses the above field equations for $\lambda, \tau$, one finds that
\begin{equation}
(\mathcal{E}_+ + \mathcal{E}_-) \Delta \mathcal{E}_\pm = 2 \nabla \mathcal{E}_\pm \cdot \nabla \mathcal{E}_\pm
\end{equation}  The equations determine the critical points of the functional
\begin{equation}
I[\mathcal{E}+, \mathcal{E}_-]:= \int_{M/K} h^{ab} \frac{\nabla_a \mathcal{E}_+ \nabla_b \mathcal{E}_-}{(\mathcal{E}_+ + \mathcal{E_-})^2} \; \td \text{Vol}(h)
\end{equation} which defines a harmonic map from the transverse space $(M/K, h) \to (N,G)$ where the target space metric is
\begin{equation}
G_{AB} \td X^A \td X^B = \frac{\td X^1 \td X^2}{(X^1 + X^2)^2}
\end{equation} This is, curiously enough, the Lorentzian metric on $N=$ AdS$_2$, the maximally symmetric spacetime with negative curvature. The underlying manifold is $\mathbb{R}^2$, the Ricci curvature scaled such that $\text{Ric}(G) = -4G$ . 

\section{Smoothness of the functions $(x,y)$}\label{Smoothness of the functions (x,y)}
Consider flat space $(\mathbb{R}^4, \delta)$ written as $\mathbb{R}^2 \times \mathbb{R}^2$. It is obvious one can write the metric as
\begin{equation}\label{Euclidean}
g = \td r_1^2 + r_1^2 \td\phi_1^2 + \td r_2^2 + r_2^2 \td \phi_2^2
\end{equation} with $r_i > 0, \phi_i \sim \phi_i + 2\pi$.  By passing to Cartesian coordinates $(x_1, x_2, x_3 , x_4)$ one can see $r_1^2$ and  $r_2^2$ are smooth functions.  Now define functions $(\rho,z)$ by
\begin{equation}
z := \frac{1}{2} (r_2^2 - r_1^2), \qquad \rho  := \sqrt{ r_1^2 r_2^2}, 
\end{equation} which can be inverted: 
\begin{equation}
r_1^2 = \sqrt{\rho^2 + z^2} - z, \qquad r_2^2 = \sqrt{\rho^2 + z^2} + z. 
\end{equation} Note  that the function $z$, as well as the following functions of $(\rho,z)$ are smooth:
\begin{equation}
\sqrt{\rho^2 + z^2} = \frac{1}{2} \left( r_1^2 + r_2^2 \right) \qquad \rho^2 = r_1^2 r_2^2
\end{equation} The coordinates $(\rho,z, \phi_i)$ form Weyl coordinates for $\mathbb{R}^4$, which has a single centre at the point $(\rho,z) = (0,0)$.  In this chart,  the Euclidean metric  \eqref{Euclidean} takes the form
\begin{equation}
g = \frac{\td\rho^2 + \td z^2}{2 \sqrt{\rho^2 + z^2}} + (\sqrt{\rho^2 + z^2} - z) \td \phi_1^2 + (\sqrt{\rho^2 + z^2} + z) \td\phi_2^2 . 
\end{equation} Note that $\rho^2 = \det g(\partial_i, \partial_j)$,  $\partial / \partial \phi_1$ vanishes on the rod $\rho=0, z>0$ and $\partial / \partial \phi_2$ vanishes on the rod $\rho=0, z < 0$. 

Consider a general Riemannian manifold with $U(1)^2$ isometry, the above coordinates can be introduced in a neighbourhood of a corner point $(\rho=0, z = z_i)$ with the replacement $z \to z-z_i$ in the formulas above.   The functions $\sqrt{\rho^2 + (z-z_i)^2}$ are seen to be smooth functions, including in particular in a neighbourhood of such points.   

We now consider the Chen-Teo metric~\eqref{g}.  On the open set on which the local metric is defined, the functions $x,y$ are smooth by definition. Now define the functions
\begin{equation}
R_1 \equiv \sqrt{\rho^2 + (z-z_1)^2}, \qquad R_2 \equiv \sqrt{\rho^2 + (z-z_2)^2}, \qquad R_3 \equiv \sqrt{\rho^2 + (z-z_3)^2},
\end{equation} which, as we have seen, are smooth functions. A computation reveals that they satisfy the constraint
\begin{equation}
(z_1 - z_2)R_3^2 + (z_3 - z_1)R_2^2 + (z_2 - z_3) R_1^2 + (z_1 - z_2)(z_3 - z_1)(z_2 - z_3) =0.
\end{equation} Next define the constants $n_i, f_i$:
\begin{align}
n_1 &= -\frac{2 \xi -1}{(1- \xi)(1 - 2\xi^2)^2(1 - 2\xi + 2 \xi^2)}, \qquad n_2 = -\frac{1 - 2\xi + 2\xi^2}{(1 - \xi)(2\xi -1)(1 - 2\xi^2)^2} \\
n_3 & = -\frac{4(1 - \xi)\xi^2}{(2\xi -1)(2\xi^2-1)^2(1 - 2\xi + 2\xi^2)} \\
f_1 & =\frac{1}{(1-\xi)(2\xi^2-1)^2(1 - 2\xi + 2\xi^2)} \qquad f_2 = \frac{1}{(1-\xi)(2\xi-1)(2\xi^2-1)^2} \\
f_3& = \frac{1}{\xi(2\xi-1)(2\xi^2-1)^2(1 - 2\xi + 2\xi^2)}
\end{align} Note that the $f_i > 0$ in our parameter range $1/2 < \xi < 1/\sqrt{2}$.   The relations \eqref{rho}, \eqref{z} can be inverted to yield
\begin{align}
x &=  \frac{2(n_1R_3 + n_2 R_2 + n_3 R_1) +1}{2(f_1 R_3 + f_2 R_2 + f_3 R_1)} \\
y & = \frac{2(n_1 R_3 + n_2 R_2 + n_3 R_1)-1}{2(f_1 R_3 + f_2 R_2 + f_3 R_1)}
\end{align} This demonstrates that $x,y$ are smooth functions on $M$ (note in particular since $f_i > 0$, the denominators in the above expressions cannot vanish). 

\section{A homeomorphism from $M$ to $ \C P^2 \setminus S^1$}\label{A homeomorphism from M}

We can convert between the Chen-Teo coordinates on $M \cong \C P^2 \setminus S^1$, to the action-angle coordinates on $\C P^2$ familiar in symplectic geometry, and to the homogeneous coordinates of complex geometry.  The idea is that since the orbit space $M/T$ is a rectangle in the $x$-$y$-coordinates,  and is a triangle in action coordinates, we need only find a change of coordinates to transform that rectangle to a triangle.

Consider the linear change of variable  $ \tilde{x} :=  \frac{x_2-y}{x_2 -x_1}$ and $\tilde{y} := \frac{x-x_2}{x_3 -x_2}$ so that $\tilde{x}$ and $\tilde{y}$ range between $0$ and $1$ with ``infinity" lying at the origin.  In coordinates 
\begin{align}
u &:= \frac{(1-\tilde{x}^2)(1+\tilde{y}^2)}{2}  & v  & := \frac{ (1+ \tilde{x}^2)(1 -\tilde{y}^2)}{2}
\end{align}
the orbit space is the triangular region with vertices $(u,v) = (0,0), (1,0), (0,1)$  and infinity located at $(\frac{1}{2}, \frac{1}{2})$.  Then we produce a homeomorphism from $M$ to $\C P^2 \setminus S^1$ in homogeneous coordinates sending 
$$(u,v, \phi^2, \phi^3) \mapsto [ \sqrt{1-u-v}: e^{i\phi^2}  \sqrt{u} : e^{i\phi^3} \sqrt{v}].$$
We note however that this transformation is continuous, but not smooth along the bolts.

\section{$L^2$ energy for $T^2$-invariant (anti-)self-dual harmonic forms}\label{LThf}

In this appendix we present an alternative method for computing the $L^2$-energy for the (anti-)self-dual harmonic forms constructed above. The approach applies generally to any setting in which such harmonic forms are invariant under the action of a torus isometry acting on an Ricci-flat,  AF Riemannian manifold $(M,g)$. The two-dimensional subspaces of the tangent space orthogonal to the generators of this torus action are assumed to be integrable (a sufficient condition for this to hold is that $\text{Ric}(g) =0$).   We will derive a general formula for the energy and then consider in detail the specific case of the Chen-Teo gravitational instanton. 
\subsection{General rod strucutre}
Recall that for (anti-)self-dual forms invariant under a single local isometry, the computations of \S \ref{harmonic} showed that
\begin{equation}
\omega \wedge \star \omega = \frac{2}{\lambda} |\td \alpha|^2_g \td \text{Vol}(g) \; . 
\end{equation} We can rewrite the PDEs \eqref{scalarSD}, \eqref{scalarASD} in the form of a conserved current: 
\begin{equation}\label{conserved}
\td \star_h \left[ \frac{\td \alpha}{\lambda} \mp \frac{\alpha \td \zeta}{\lambda^2} \right] =0
\end{equation} where we used the fact $\td \star_h (\lambda^{-2} \td \zeta)=0$ which follows from the reduced system \eqref{Ricciflat}.  Now suppose that $\omega$ is invariant under a second  $U(1)$ isometry generated in a coordinate chart by $\partial_\phi$.  We may write the orbit space metric $h$ in the form
\begin{equation}
h = g_2 + \rho^2 \td \phi^2,  
\end{equation} where $g_2$ is the metric induced on the integrable surfaces orthogonal to the $U(1) \times U(1)$ action.  Let $\beta$ be any 1-form satisfying $i_{\partial_\phi} \beta =0$.  We then have $\star_h \beta = \td \phi \wedge (\rho \star_2 \beta)$ where $\star_2$ is the Hodge dual operation with respect to $g_2$.  Substituting $\beta = \td \alpha$,  \eqref{conserved}  is equivalent to
\begin{equation}\label{current}
\td  \star_2 \rho \left[ \frac{\td \alpha}{\lambda}  \mp \frac{\alpha \td \zeta}{\lambda^2}\right] =0
\end{equation}  The $L^2$ energy can be expressed as an integral over $B$:
\begin{equation}
\frac{1}{2} ||\omega||^2_{L^2}=\frac{1}{8\pi^2} \int_M \frac{2}{\lambda} |\td \alpha|^2_g \td \text{Vol}(g) =   \int_B\frac{ |\td \alpha|^2_2}{\lambda} \rho  \td \text{Vol}(g_2) = -  \int_B  \lambda^{-1}\rho\,  \td \alpha \wedge \star_2 \td \alpha
\end{equation} where $B = M \setminus T^2$  and $\rho^2$ is the determinant of $g$ restricted to Killing vectors fields $\ell_1, \ell_2$ generating periodic flow with associated $2\pi-$periodic coordinates $(\phi^1, \phi^2)$.   The $-$ sign arises from our convention for the Hodge dual, i.e. $\star_2 \td x^a = \epsilon_{b}^{~a} \td x^b$ where $\epsilon_{ab}$ represents the volume form associated to $g_2$. Then
\begin{equation}\begin{aligned}
\frac{1}{2}  ||\omega||^2_{L^2} &= -  \int_B \td \left[ \frac{\rho}{\lambda} \alpha \star_2 \td \alpha \right] - \alpha \td \left[ \frac{\rho}{\lambda} \star_2 \td \alpha \right]  = -   \int_B \td \left[ \frac{\rho}{\lambda} \alpha \star_2 \td \alpha \right]  \mp \alpha \td \left[ \frac{\alpha \rho}{\lambda^2} \star_2 \td \zeta \right] \\
& = - \int_B  \td \left[ \frac{\rho}{\lambda} \alpha \star_2 \td \alpha  \mp \frac{ \alpha^2 \rho}{2 \lambda^2} \star_2 \td \zeta \right]
\end{aligned}
\end{equation} where in the first and second lines we used \eqref{current} and the condition $\td (\rho \lambda^{-2} \star_2 \td \zeta)=0$.  Then Stokes' theorem applied to the two-dimensional manifold $B$ with boundary $\partial B$ and asymptotic end $\partial B_\infty$ gives
\begin{equation}\label{bdyint}
\frac{1}{2} ||\omega||^2_{L^2} = -\int_{\partial B \cup \partial B_\infty} \left[ \frac{\rho}{\lambda} \alpha \star_2 \td \alpha \mp \frac{\alpha^2 \rho}{2 \lambda^2} \star_2 \td \zeta \right].
\end{equation}  As discussed in \S 4, $B$ can be given global coordinates $(\rho,z)$ with $\rho > 0, z \in \mathbb{R}$. Note that $\partial B$, upon which the torus action degenerates, corresponds to  $\rho =0$. Hence there should be no contributions  to $E$ from the interior of rods. However at the corner points, $\lambda$ has simple zeroes.  To evaluate the boundary integral, we take a semi-circular contour around each corner point and then shrink these contours to zero. Finally at the asymptotic end, note that $\rho \sim r \sin \theta$ whereas $\td\alpha = O(1/r)$ so these contributions will converge. 

In a neighbourhood of a given corner point $z = z_i$ we may introduce an adapted spherical coordinate system $(R, \theta, \phi_R, \phi_L)$ centred at  $z_i$~\cite{Hollands:2008fm}.  In this chart,  full metric can be expressed as
\begin{equation}
g  = \td R^2 + R^2 \left[ \td \theta^2 + \sin^2 \theta \td \phi_R^2 + \cos^2 \theta \td \phi_L^2 \right] + O(R^4)
\end{equation} where $R > 0$, $\theta \in [0,\pi/2]$,  $R  = 0$ corresponds to the centre and $\theta =0$ corresponds to the symmetry axis of $\partial_{\phi_R}$ and similarly for $\theta = \pi/2$ and $\partial_{\phi_L}$.   The Killing fields $\partial_{\phi_R}, \partial_{\phi_L}$ vanish to the right ($z> z_i$) and left $z < z_i$ respectively of the corner point. We may express the Killing vector field $K$ in terms of this basis:
\begin{equation}
K = c_1 \frac{\partial}{\partial \phi_R} + c_2 \frac{\partial}{\partial \phi_L}
\end{equation} for some constants $c_1, c_2$. 
Now as $R \to 0$,
\begin{equation}
\lambda = |K|^2 = R^2 \left[ c_1^2 \sin^2\theta + c_2^2 \cos^2\theta + O(R^2) \right], \qquad \rho = R^2 \sin \theta \cos \theta + O(R^2)
\end{equation} and we note that the term in the square brackets is non-vanishing.  Consider the first term in \eqref{bdyint}.  We know $\alpha$ is a smooth function, so we can write
\begin{equation}
\td \alpha = \partial_R \alpha \td R + \partial_\theta \alpha \td \theta, \qquad \star_2 \td \alpha = -R \partial_R \alpha \td \theta + \frac{1}{R} \partial_\theta \alpha \td R
\end{equation}  Now consider an semi-circular integration contour about the centre with  $R = \epsilon$ fixed and $\theta \in (0,\pi/2)$.  Note that 
\begin{equation}
R \partial_R \alpha = x^i \partial_{x^i} \alpha, 
\end{equation} where $x^i$ are Cartesian coordinates and $R^2 = \sum^4_{i=1} (x^i)^2$.  This is obviously smooth and the integral is $O(\epsilon)$ and hence vanishes as  the semicircle shrinks to zero.  Next let us examine the term involving the twist potential. In this coordinate chart, we find that
\begin{equation}
\td \zeta = \star_g (K \wedge \td K) = 2 c_1 c_2 R \td R
\end{equation} where we chose the orientation so that $(R, \theta, \phi_R, \phi_L)$ is positively oriented.  It follows that $\star_2 \td \zeta = -2 c_1 c_2 R^2 \td \theta$. Therefore we have, as $R \to 0$
\begin{equation}
\frac{\alpha^2 \rho}{2 \lambda^2} \star_2 \td \zeta = -\frac{\alpha^2(0)c_1 c_2 \sin\theta \cos\theta}{(c_1^2 \sin^2\theta + c_2^2 \cos^2\theta)^2} \td \theta + O(R)
\end{equation} which we see is $O(1)$ as we integrate on a semicircle with $R = \epsilon \to 0$.  Using the fact that
\begin{equation}
\int_0^{\frac{\pi}{2}} \frac{c_1 c_2 \sin\theta \cos\theta}{(c_1^2 \sin^2\theta + c_2^2 \cos^2\theta)^2} = \frac{1}{2 c_1 c_2}
\end{equation} we find, taking into the orientation of $\partial B$,  that
\begin{equation}
 -  \int_{\partial B} \left[ \frac{\rho}{\lambda} \alpha \star_2 \td \alpha \mp \frac{\alpha^2 \rho}{2 \lambda^2} \star_2 \td \zeta \right] = \pm \frac{1}{2}   \sum_{z_i}  \frac{\alpha(z_i)^2 }{c^i_R c^i_L} . 
 \end{equation} where $z_i$ refer to the corner points with associated frequencies $c^i_R, c^i_L$ and the orientation is chosen so that one integrates in the order of increasing $z$. 

 Finally, we consider the contributions to \eqref{bdyint} arising from the asymptotic end.  In the asymptotic region, the metric  of an AF gravitational instanton approaches
\begin{equation}\label{genAF}
g = A^2( \td \phi^2)^2 \td r^2 + r^2( \td \vartheta^2 +  \sin^2\vartheta \left[ \td \phi^1 + B \td \phi^2\right]^2)
\end{equation} for certain constants $A >0, B$ (note that, as in the specific Chen-Teo case discussed at length above, we have chosen $\ell_1$ to degenerate on the asymptotic axes of symmetry) and $\vartheta \in (0,\pi)$.   In particular we read off, as $r \to \infty$ that $\rho \to A r \sin \vartheta$.  For convenience we will normalize the Killing field $K$ so that $\lambda \to 1$ as $r \to \infty$.  The integral over the asymptotic boundary is easily obtained. 
We summarize these computations with the following proposition. 
\begin{prop} Suppose $(M,g)$ is an AF gravitational instanton admitting a torus action as isometries such that in the asymptotic region, $g$ approaches the model metric \eqref{genAF}.  Normalize the associated Killing vector field $K$ such that that $|K| \to 1$ in the asymptotic region and let $\zeta$ be the associated twist potential.  Let $\omega$ be an (anti-)self-dual harmonic 2-form generated by the smooth function $\alpha$ according to \eqref{scalarSD}, \eqref{scalarASD}.  Then its $L^2$ energy can be expressed as 
\begin{equation}\begin{aligned}\label{energy}
\frac{1}{2}||\omega||^2_{L^2} =\frac{1}{8\pi^2} \int_M \omega \wedge \star \omega &= 2 A \lim_{r \to \infty} \int_0^\pi r^2 \sin \vartheta \left(2 \alpha \partial_r \alpha \mp \alpha^2 \partial_r \zeta\right) \mathrm{d}  \vartheta
 \pm \frac{1}{2}   \sum_{z_i}  \frac{\alpha(z_i)^2 }{c^i_1 c^i_2} . 
\end{aligned}
\end{equation} where the upper and lower signs refer to self-dual and anti-self-dual harmonic forms respectively and $z_i$ refer to corner points with associated frequencies $c_1^i, c^i_2$ with respect to the Killing vector field $K$. 
\end{prop} \noindent We note that for the energy to converge, it is sufficient to require $\alpha = O(r^{-2})$ as $r \to \infty$. 

\subsection{Chen-Teo instanton}
We now turn to the computation of the $L^2$ energies associated to the self-dual harmonic form $\omega_+$ and the anti-self-dual harmonic forms $\omega_-, \omega_2$ in the Chen-Teo gravitational instanton.  We choose $K = \partial_\tau$ (note that this does not satisfy the normalization condition $|K|^2 \to 1$ in the asymptotic region, although it is straightforward to modify \eqref{energy}).  There are 3 corner points $z_i$ to consider.  The vector fields $\partial/ \partial \phi_R, \partial / \partial \phi_L$ associated to each $z_i$  are fixed by the requirement that $\td \phi_R \wedge \td \phi_L$ has positive orientation with respect to $\td \phi^1 \wedge \td \phi^2$.  We find that
\begin{equation}
\begin{aligned}
&z_3: \quad \frac{\partial}{\partial \phi_R} = \ell_1, \qquad \frac{\partial}{\partial \phi_L} = \ell_2 \\
&z_2: \quad \frac{\partial}{\partial \phi_R} = \ell_2, \qquad \frac{\partial}{\partial \phi_L} = -\ell_1 + \ell_2 \\
&z_1: \quad \frac{\partial}{\partial \phi_R} = \ell_1-\ell_2, \qquad \frac{\partial}{\partial \phi_L} = \ell_1
\end{aligned} 
\end{equation} where $(\ell_1, \ell_2)$ are defined by \eqref{ellinspi}.  We may then calculate the associated `frequencies' 
\begin{equation}
\begin{aligned}
c^3_R &= -\frac{(1 - \xi)^2}{2\sqrt{\kappa} \xi^2}, \qquad c^3_L = \frac{(1 - 2\xi + 2\xi^2)^2}{8\sqrt{\kappa} \xi^4} \\
c^2_R & = \frac{(1 - 2\xi)^2}{8\sqrt{\kappa} \xi^4}, \qquad c^2_L = -c^3_R \\
c^1_R & = -c^3_L, \qquad c^1_L  = c^2_R
\end{aligned}
\end{equation} We record the values
\begin{equation} \begin{aligned}
\alpha_+ (z_1)&  = -\alpha_-(z_1) =-\frac{1 - 3\xi + 2\xi^2 - 2\xi^3}{\xi(1-4\xi^4)} , \qquad  \alpha_+(z_2) = -\alpha_-(z_2)  =  -\frac{1 - 2\xi - 4\xi^3 + 4\xi^4}{2\xi^2(1 - 4\xi^4)}, \\
 \alpha_+(z_3) &=-\alpha_-(z_3)  =  \frac{1 - 2\xi + 6\xi^2 - 4\xi^3}{2\xi^2(1 - 4\xi^4)}\\
\alpha_2(z_1) & = -\frac{1}{4\xi^4(1-\xi)(1-4\xi^4)}, \qquad \alpha_2(z_2) = -\frac{1}{2\xi^3(1-4\xi^4)(1-2\xi + 2\xi^2)}, \\  \alpha_2(z_3) &= \frac{1}{2\xi^3(1 - 2\xi)(1-4\xi^4)}
\end{aligned}
\end{equation}  Next, we consider the contributions from the asymptotically flat end.  Using the appropriate values of $A, B$  one finds
\begin{equation}
\rho  =  \frac{8 \sqrt{\kappa} \xi^4}{\sqrt{1 - 4\xi^4} (1 - 2\xi + 2 \xi^2)^2} \cdot r \sin\theta + O(1). 
\end{equation} Furthermore it is straightforward to verify that as $r \to \infty$,  
\begin{equation}
\partial_r \zeta = -\frac{4 \kappa \xi^2 ( 1 - \xi + 2 \xi^2) \cos\theta}{(1 - \xi)(1 - 2\xi)(1 - 4\xi^4)} \frac{1}{r^3} + O(r^{-4})
\end{equation} and hence $r^2 \partial_r \zeta = O(1/r)$ and does not contribute to the integral over the asymptotic boundary as $r \to \infty$.  Similarly, 
\begin{equation}
\begin{aligned}
\alpha_+ \partial_r \alpha_+ &= \frac{2\sqrt{\kappa} \sqrt{1 - 4 \xi^4}}{(1 - 2 \xi^2)^3(1 + 2 \xi^2) r^2} + O(r^{-3}) \\
\alpha_- \partial_r \alpha_- & = O(1/r^3), \qquad \alpha_2 \partial_r \alpha_2 = O(1/r^3)
\end{aligned}
\end{equation} Recall that $\alpha_- , \alpha_2$ both vanish in the asymptotic region as $O(1/r)$.  Applying the formula \eqref{energy}  we find
\begin{equation}
\frac{1}{2}||\omega_+||^2_{L^2} = \frac{1}{2}||\omega_-||^2_{L^2}= \frac{16 \kappa \xi^4}{(1 - 2\xi^2)^2(1 - 2\xi + 2\xi^2)^2}
\end{equation} and
\begin{equation}
\frac{1}{2}||\omega_2||^2_{L^2} = \frac{ \kappa}{(1- \xi)^2(1 - 2\xi)^2(1 - 4\xi^4)^2(1 - 2\xi + 2 \xi^2)^2}
\end{equation} Finally, let us consider the computation of
\begin{equation}
q_{ij} :=\frac{1}{8\pi^2}\int_M \omega_I \wedge \star \omega_J
\end{equation} where $I=1,2$ label the anti-self-dual forms. For simplicity we take as our basis $\omega_-, \omega_2$ generated by $\alpha_-, \alpha_2$ respectively.  We have already computed the diagonal components  $q_{11}, q_{22}$. An easy way to compute $q_{12}$ is to use linearity of the PDE satisfied by $\alpha$.  Namely,  if $\alpha$ and $\beta$ are functions which give rise to anti-self-dual  harmonic forms $\omega_\alpha, \omega_\beta$, then $\alpha + \beta$ will give rise to another anti-self-dual harmonic form $\omega_{\alpha + \beta}$.  We have the `parallelogram identity' 
\begin{equation}
\frac{1}{8\pi^2}\int_M \omega_\alpha \wedge \star \omega_\beta = \frac{1}{2} \left [ ||\omega_{\alpha + \beta}||^2_{L^2} - ||\omega_\alpha||^2_{L^2} - ||\omega_\beta||^2_{L^2}\right]. 
\end{equation} We find
\begin{equation}
\frac{1}{8\pi^2}\int_M \omega_- \wedge \star \omega_2 = -\frac{16 \kappa \xi^3}{(1 -\xi)(1 - 2\xi)(1 - 4\xi^4)^2(1 - 2\xi + 2\xi^2)^2}
\end{equation} This completes the calculation of $q_{ij}$. 
\begin{remark}
Consider the harmonic two-form $\td K$.  We may express its energy as 
\begin{equation}\label{L2K}
\frac{1}{2}||\td K||^2_{L^2} = \frac{1}{8\pi^2}\int_M \td K \wedge \star \td K = \frac{1}{8} (||\omega_+||^2_{L^2} + ||\omega_-||^2_{L^2}) = \frac{8 \kappa\xi^4}{(1 - 2\xi + 2\xi^2)^2 (1 - 2\xi^2)^2}
\end{equation}   Alternatively,  we may compute this integral by using 
\begin{equation}
\int_M \td K \wedge \star \td K = \int_M \td \left( K \wedge \star \td K \right) = \int_{\partial M} K \wedge \star \td K
\end{equation} and evaluating the integral over the asymptotic boundary of $M$.  It is an exercise to verify that this yields \eqref{L2K}.
\end{remark}

\end{document}